\documentclass[11pt]{amsart}
\usepackage{amssymb,epsfig,newlfont,amsmath}

\usepackage[english]{babel}
\newtheorem{theorem}{Theorem}[section]
\newtheorem{lemma}[theorem]{Lemma}
\newtheorem{corollary}[theorem]{Corollary}
\newtheorem{proposition}[theorem]{Proposition}

\theoremstyle{definition}
\newtheorem{definition}[theorem]{Definition}
\newtheorem{example}[theorem]{Example}

\theoremstyle{remark}
\newtheorem{remark}[theorem]{Remark}

\numberwithin{equation}{section}

\newcommand{\Z}{\mathbb{Z}}
\newcommand{\Q}{\mathbb{Q}}
\newcommand{\R}{\mathbb{R}}
\newcommand{\C}{\mathbb{C}}
\newcommand{\A}{\mathbb{A}}

\newcommand{\sG}{\mathsf{G}}
\newcommand{\sP}{\mathsf{P}}
\newcommand{\sQ}{\mathsf{Q}}
\newcommand{\sR}{\mathsf{R}}

\newcommand{\sL}{\mathsf{L}}
\newcommand{\sN}{\mathsf{N}}
\newcommand{\sM}{\mathsf{M}}
\newcommand{\sS}{\mathsf{S}}
\newcommand{\sH}{\mathsf{H}}
\newcommand{\sB}{\mathsf{B}}
\newcommand{\sT}{\mathsf{T}}

\newcommand{\Ker}{\mathrm{Ker}}

\newcommand{\diag}{\mathrm{diag}}

\makeatletter
\def\revddots{\mathinner{\mkern1mu\raise\p@
\vbox{\kern7\p@\hbox{.}}\mkern2mu
\raise4\p@\hbox{.}\mkern1mu\raise7\p@\hbox{.}\mkern1mu}}
\makeatother

\numberwithin{equation}{section}
\numberwithin{figure}{section}
\numberwithin{table}{section}

\begin{document}

\title[Eisenstein series and the top degree cohomology of arithmetic groups]{Eisenstein series and the top degree cohomology of arithmetic subgroups of $SL_n/\Q$}

\author{Joachim Schwermer}
\address{Faculty of Mathematics, University Vienna, Oskar-Morgenstern-Platz 1, A-1090 Vienna, Austria  \and Max-Planck-Institute 
for Mathematics, 
Vivatsgasse 7, D-53111 Bonn, Germany.}
\email{Joachim.Schwermer@univie.ac.at}

\thanks{ The author  gratefully acknowledges the support of the Max-Planck-Institute for Mathematics, 
Bonn, in the academic year 2019/20}

\subjclass[2010]{Primary 11F75; Secondary 11F67, 11F70, 22E40,
22E55}

\date{March 5, 2020}



\begin{abstract}
The cohomology $H^*(\Gamma, E) $ of a torsion-free arithmetic subgroup $\Gamma$ of the  special linear  $\Q$-group $\sG = SL_n$ may be interpreted in terms of the automorphic spectrum of $\Gamma$. Within this framework, there is a decomposition of the cohomology into the cuspidal cohomology
and the  Eisenstein cohomology 
 The latter space is  decomposed according to the 
classes $\{\sP\}$ of associate proper parabolic $\Q$-subgroups of $\sG$. Each summand $H^*_{\mathrm{\{P\}}}(\Gamma, E)$ 
 is built up by Eisenstein series (or residues of such) attached to cuspidal automorphic forms
 on the Levi components of elements in $\{\sP\}$. 

The cohomology  $H^*(\Gamma, E) $ vanishes above the degree given by the cohomological dimension $\mathrm{cd}(\Gamma) = \frac{n(n-1)}{2}$. 
We are concerned with  the internal structure of the cohomology in this top degree. On the one hand, we explicitly describe the associate classes $\{\sP\}$ for which the corresponding summand $H^{\mathrm{cd}(\Gamma)}_{\mathrm{\{\sP\}}}(\Gamma, E)$ vanishes. On the other hand, in the remaining cases of associate classes 
we construct various families of non-vanishing Eisenstein cohomology classes which span $H^{\mathrm{cd}(\Gamma)}_{\mathrm{\{\sQ\}}}(\Gamma, \C)$. 

Finally, in the case of a principal congruence subgroup $\Gamma(q)$, $q = p^{\nu} > 5$, $p\geq 3$ a prime, we give lower bounds for the size of these spaces if not even a precise formula for its dimension for certain associate classes $\{\sQ\}$.

\end{abstract}

\maketitle

\section{Introduction}\label{intro}

\subsection{ } Given  a connected semi-simple algebraic $\Q$-group $\sG$ of $\textrm{rk}_\Q \sG > 0$, the group of real valued points $G = \sG(\R)$ is a  real Lie group. 
Let $K$ be a  maximal compact subgroup of $G$. The homogeneous space $G/K = X$, the symmetric space associated to $G$, is diffeomorphic to $\R^{d(G)}$, where $\mathrm{d}(G) = \dim G - \dim K$, hence $X$ is contractible.

An arithmetic subgroup $\Gamma$ of $\sG$ is a discrete subgroup of $G$, and $\Gamma$ acts properly discontinuously on $X$. If $\Gamma$ is torsion-free, the action is free, and the quotient space $\Gamma\setminus X$ is a smooth manifold of dimension $\mathrm{d}(G)$. 
Since the underlying algebraic $\Q$-group $\sG$ has positive $\Q$-rank, $\Gamma\setminus X$ is non-compact but of finite volume. It can be viewed as the interior of a natural compactification with boundary,  due to Borel and Serre \cite{BS}.  
This compactification is obtained  as the quotient under $\Gamma$ of  a $\sG(\Q)$-equivariant partial compactification $\overline{X}$.
 
The inclusion $\Gamma \setminus X \longrightarrow \Gamma \setminus \overline{X}$ is a homotopy equivalence.
The boundary $\partial(\Gamma \setminus \overline{X})$ is glued together out of faces $e'(Q)$,  where $\sQ \neq \sG$ ranges over a  set of representatives for the finitely many  $\Gamma$-conjugacy classes of  proper parabolic $\Q$-subgroups of $\sG$.  The codimension of a given face  in $\Gamma \setminus \overline{X}$ is the parabolic $\Q$-rank of $\sQ$. 
The closures of the faces $e'(Q)$ form a closed cover of the boundary $\partial(\Gamma \setminus \overline{X})$ whose nerve is the quotient under the natural action of $\Gamma$ of the Tits building $\mathcal{T}_{\sG}$ of proper parabolic $\Q$-subgroups of $\sG$. 
 
 Let $(\nu, E)$ be a finite-dimensional irreducible representation of the real Lie group $G$ on a complex vector space.
 Then the cohomology $H^*(\Gamma, E)$ is isomorphic to the  cohomology $H^*(\Gamma \setminus X, E)$ [computed, for example, via the DeRham complex of $E$-valued $\Gamma$-invariant differential forms on $X$]. Interpreting these latter groups  in terms of the automorphic spectrum of $\Gamma$,  by \cite{F1} resp. \cite{F--S}, there is a decomposition of the cohomology into the cuspidal cohomology
(i.e. classes represented by cuspidal automorphic forms for $\sG$ with respect to $\Gamma$) and the  Eisenstein cohomology (i.e. classes represented by forms arising from Eisenstein series for $\sG$ with respect to $\Gamma$).  
 The latter space is  decomposed according to the 
classes $\{\sP\}$ of associate proper parabolic $\Q$-subgroups of $\sG$, that is, eventually we obtain
 \begin{equation}\label{autodecomp}
H^*(\Gamma \setminus X, E) = H^*_{\textrm{cusp}}(\Gamma \setminus X, E) \oplus \bigoplus_{\{\sP\}}H^*_{\{\sP\}}(\Gamma \setminus X, E).
\end{equation}
Each summand indexed by  $\{\sP\}$
 is built up by Eisenstein series (appropriate derivatives or residues of such) attached to cuspidal automorphic forms
 on the Levi components of elements in $\{\sP\}$. Note that an associate class $\{\sP\}$ falls into finitely many $\sG(\Q)$-conjugacy classes.
 
 The role of the theory of Eisenstein series, especially their residues, in describing the structure of the cohomology $H^*(\Gamma \setminus X, E)$ was dealt with by various authors, among them,  Harder \cite{Ha3}, Franke \cite{F1}, J-S Li \cite{LS3}, and the author of this paper in \cite{S2}, \cite{S4}, \cite{S10}, or, jointly with Grbac, in \cite{gs2}, \cite{gs4}. Nevertheless, the full elucidation of the structure of the Eisenstein cohomology of such arithmetic quotients is still a major outstanding challenge.

 The cohomology  $H^*(\Gamma \setminus X, E) $ vanishes above the degree given by the cohomological dimension $\mathrm{cd}(\Gamma)$. Therefore it is natural, possibly even technically necessary, to start with this extreme non-trivial case.
 It is the aim of this paper, in the case of a congruence subgroup $\Gamma$ of the $\Q$-group $\sG = SL_n$,
to unfold the automorphic  view on the cohomology group $H^{\textrm{cd}(\Gamma)}(\Gamma \setminus X, E)$ and to gain some insight into  the internal structure of the cohomology in this top degree 
$\mathrm{cd}(\Gamma)  = \frac{n(n-1)}{2}$. This includes some quantitative results regarding the size of these cohomology spaces.

\subsection{ } With regard to associated classes $\{\sP\}$, we use that the $\sG(\Q)$-conjugacy classes of parabolic $\Q$-subgroups of $\sG$ are in bijection with the subsets $J \subset \Delta_{\Q}$ of the set of simple roots determined by the Levi decomposition $\sP_0 = \sL_0\sN_0$ of the parabolic subgroup of upper triangular matrices where the maximal split torus $\sL_0$ is the group of diagonal matrices in $\sG$. If $\sP_J = \sL_J\sN_J$ is the parabolic $\Q$-subgroup which corresponds to $J$, the roots of $\sL_J$ are those roots whose simple components are in $J$.

We denote by $\mathcal{J}^{\mathrm{cd}}$ the family of non-empty subsets $J \subset \Delta_{\Q} = \{\alpha_1, \ldots, \alpha_{n-1}\}$ subject to the condition that if $\alpha_i, \alpha_{i+1} \in J$, for some $i \in \{1, \dots, n-2\}$, then $\alpha_{i+2} \notin J$. As a consequence of the following result this family, together with the empty set, eventually describe the only sources for the possible existence of  Eisenstein cohomology classes in the degree of the cohomological dimension.

\begin{proposition} Let $\{\sP\}$ be an associate class of proper parabolic $\Q$-subgroups of $SL_n/\Q$. If   $\sP$ is neither the class $\{\sP_0\}$ of minimal parabolic $\Q$-subgroups
nor $\sP$ is $\sG(\Q)$-conjugate to a standard parabolic subgroup $\sP_J$ where the defining set $J$ is an element in $\mathcal{J}^{\mathrm{cd}}$  then the corresponding summand $H^{\mathrm{cd}(\Gamma)}_{\{\sP\}}(\Gamma \setminus X, E)$ in the decomposition (\ref{autodecomp}) vanishes in degree $\mathrm{cd}(\Gamma)$.
\end{proposition}

 In view of this result, we only have to deal  with the summands $H^{\mathrm{cd}(\Gamma)}_{\{\sP\}}(\Gamma \setminus X, E)$ in the decomposition (\ref{autodecomp}) where $\sP$ is $\sG(\Q)$-conjugate to a standard parabolic subgroup $\sP_J$ with $J \in \mathcal{J}^{\mathrm{cd}}$, or to $\sP_{\emptyset} = \sP_0$. As it will turn out,  in all these cases, the corresponding summand does not vanish for a congruence subgroup $\Gamma$ of sufficiently high  level.

 However, here is a more precise result concerning a subfamily of index sets in  $\mathcal{J}^{\mathrm{cd}}$. 
\begin{theorem}Let $\sP$ be a parabolic $\Q$-subgroup of $\sG/\Q = SL_n/\Q$, $n \geq 3$, whose $\sG(\Q)$-conjugacy class is represented by a standard parabolic $\Q$-subgroup $\sP_J$ indexed by the  set $J = \{\alpha_{i_1}. \ldots, \alpha_{i_r} \} \subset \Delta_{\Q}$, $J \neq \emptyset$, subject to the conditions that if $\alpha_i, \alpha_j \in J$ then $\vert i - j \vert \geq 2.$ 
Given a torsion-free principal congruence  subgroup $\Gamma \subset \sG(\Q)$,  the space
\begin{equation}
H^{\mathrm{cd}(\Gamma)}_{\{\sP\}}(\Gamma \setminus X; \C) = \bigoplus_{\sQ \in \{\sP\}/\Gamma}  H^{\mathrm{cd}(\Gamma)}(\Gamma \setminus X; \C)_{e'(Q)},
\end{equation}
generated by  regular Eisenstein cohomology classes as constructed  for each of the faces $e'(Q), \sQ \in \{\sP\}$, up to $\Gamma$-conjugacy,  is non-trivial and its dimension is given by
\begin{equation}
\dim_{\C} H^{\mathrm{cd}(\Gamma)}_{\{\sP\}}(\Gamma \setminus X; \C) = \mathrm{conj}_{\sG}[\{\sP\}] 
 \cdot \mathrm{conj}_{\Gamma}(\sP) \cdot \dim_{\C}H^{\mathrm{cd}(\Gamma)}_{\mathrm{cusp}}(e'(\sP); \C)
\end{equation}
where $\mathrm{conj}_{\sG}[\{\sP\}] = \binom{n - r}{r}$ and $\mathrm{conj}_{\Gamma}(\sP)$ denotes the number of $\Gamma$-conjugacy classes of $\sP$.
\end{theorem}

Given a prime power $q = p^{\nu} > 2$, the number $\mathrm{conj}_{\Gamma(q)}(\sP)$ 
 is determined in Lemma \ref{Gamclasses}.

The remaining cases  of index sets in  $\mathcal{J}^{\mathrm{cd}}$ are dealt with in a similar way to a certain extent in Section \ref{Mixed}. Finally, in the case of a principal congruence subgroup $\Gamma(q)$, $q = p^{\nu} > 5$, $p\geq 3$ a prime, we give lower bounds for the size of these spaces if not even a precise formula for its dimension for certain associate classes $\{\sQ\}$.

 If $\{\sP\} =\{\sP_0\}$ is the associate class of minimal parabolic $\Q$-subgroups of $SL_n/\Q$, represented by the standard parabolic of upper triangular matrices, we have already given a non-vanishing result for the space $H^{\mathrm{cd}(\Gamma)}_{\{\sP_0\}}(\Gamma \setminus X, E)$ in the thesis work  \cite{S1}, \cite{S1CR} in the case of the trivial coefficient system, and in  \cite[Section 7]{S4} in the generic case. This was completed by giving a lower bound for the dimension of this space. For the sake of completeness, we briefly review this result in the case $E = \C$ in Section \ref{minpar}.

In Section \ref{Conclude}, we compare our automorphic approach to the structure of the top degree cohomology group $H^{\mathrm{cd}(\Gamma)}(\Gamma \setminus X, E)$ with prior work  where one uses the relation with the Steinberg module  and its realization as the reduced cohomology of the Tits building attached to $SL_n/\Q$.
Finally, we indicate how our results can possibly be extended to the special linear group over a number field or even to other groups than $SL_n$.

\specialsection*{Notation and conventions}

(1) Let $\mathbb{Q}$ be the field of rational numbers. We denote
by $V$ the set of places of $\mathbb{Q}$, and by $V_f$ the set of
non-archimedean places. The archimedean place is denoted by $v=\infty$. Let
$\mathbb{Q}_v$ be the completion of $\mathbb{Q}$ at $v$, and
$\mathbb{Z}_v$ the ring of integers of $\mathbb{Q}_v$ for $v\in
V_f$. Let $\mathbb{A}$ (resp.~$\mathbb{I}$) be the ring of adeles
(resp.~the group of ideles) of $\mathbb{Q}$. We denote by
$\mathbb{A}_f$ the finite adeles.

(2) The algebraic groups considered are linear. If $\sH$ is an algebraic group defined over a field $k$, and $k'$ is a commutative $k$-algebra, we denote by $\sH(k')$ the group of $k'$-valued points of $\sH$. If $\sH$ is a connected  $\Q$-group, the group of real points $\sH(\R)$, endowed with the topology associated to the one of $\R$, is a real Lie group. We denote its Lie algebra by the corresponding small gothic letter $\frak{h}$.

(3) Given a connected algebraic $\Q$-group $\sG$, we put $^0G = \cap_{\chi \in X_{\Q}(\sG))} \text{ker} \chi^2$ where $X_{\Q}(\sG)$ denotes the group of $\Q$-morphisms $\sG \rightarrow GL_1$.  The group is defined over $\Q$ and normal in $\sG$. Any character is trivial on the unipotent radical $R_u(\sG)$ of $\sG$, and, given any Levi $\Q$-subgroup  $\sL$ of $\sG$, we have $^0\sG =\; ^0\sL \ltimes R_u(\sG)$. The real Lie group $^0\sG(\R)$ contains each compact subgroup of $\sG(\R)$ and each arithmetic subgroup of $\sG$.

\vspace{0.25cm}

\section{Arithmetic quotients and their cohomology}\label{arithquo}
Let $\Gamma$ be an arithmetically defined  subgroup of a connected semi-simple algebraic $\Q$-group $\sG$ of positive $\Q$-rank. Then it can be viewed as a discrete subgroup of the real Lie group $G = \sG(\R)$. Given a maximal compact subgroup $K \subset G$, the associated  symmetric space is $X = G/K$ on which $G$, and thus $\Gamma$, acts properly. The arithmetically defined quotient space $\Gamma \setminus X$ is non-compact but has finite volume; it may be interpreted as the interior of a compact orbifold with corners. 
These finitely many corners are parametrized by the $\Gamma$-conjugacy classes of proper parabolic $\Q$-subgroups of $\sG$.
We briefly review this construction and its role within investigations regarding the cohomology of  $\Gamma \setminus X$. In particular, we describe the structure of such a corner $e'(P)$ as a fibre bundle and its cohomology. The cuspidal cohomology of $e'(P)$ is the source for the eventual construction of cohomology classes for $\Gamma \setminus X$ which are represented by Eisenstein series or residues of such.

\subsection{Parabolic subgroups, Levi subgroups, and roots}

 Let $\sG$ be a connected semi-simple algebraic group defined over
$\Q$. We assume that $\sG$ has $\Q$-rank greater than zero. Fix a minimal parabolic subgroup $\sP_0$ of $\sG$ defined over
$\Q$ and a Levi subgroup $\sL_0$ of $\sP_0$ defined over $\Q$. 
By definition, a standard parabolic subgroup $\sP$ of $\sG$ is
a parabolic subgroup $\sP$ of $\sG$ defined over $\Q$ that contains
$\sP_0$. Analogously, a standard Levi subgroup $\sL$ of $\sG$ is a Levi
subgroup of any standard parabolic subgroup $\sP$ of $\sG$ such that
$\sL$ contains $\sL_0$. A given standard parabolic subgroup $\sP$ of $\sG$
has a unique standard Levi subgroup $\sL$. We denote by $\sP = \sL_{\sP}\sN_{\sP}$ the
corresponding Levi decomposition of $\sP$ over $\Q$. When the dependency  on the parabolic subgroup is clear from the context we suppress the subscript $\sP$ from the notation.

Let $\sP$ be a parabolic $\Q$-subgroup of $\sG$,  and let $\sN$ be the unipotent radical of $\sP$. We denote by  $\kappa: \sP \longrightarrow \sP/\sN =: \sM$  the canonical projection of $\sP$ on the reductive $\Q$-group $\sM$. Let $\sS_{\sP}$ be the maximal central $\Q$-split torus of $\sM$, and let $S_P$ be denote the identity component of $\sS_{\sP}(\R)$. A subgroup $A$ of a Levi subgroup of $P = \sP(\R)$ such that $A$ is mapped under $\kappa$ isomorphically on $S_P$ is called a split component of $P$. Note that a Levi subgroup of $P$ is isomorphic via $\kappa$  to the group of real points $\sM(\R)$. Two split components of $P$ are conjugate under $N = \sN(\R)$.

Fix a maximal compact subgroup $K$ of the real Lie group $G = \sG(\R)$. We denote by $A_P$ the uniquely determined split component of $P$ which is stable under the Cartan involution $\Theta_K$ attached to $K$. Let  $M = Z_G(A_P)$ defined to be the centralizer of $A_P$ in $G$; it is the uniquely determined $\Theta$-stable Levi  subgroup of $P$. Then we have the decomposition $P = M \ltimes N$ as a semi-direct product.  Analogously we have $ P = A_P \ltimes\; ^0P$.  

The projection $\kappa$ induces  a canonical isomorphism $\mu: M \longrightarrow \sM(\R)$; we denote by $^0M$ the inverse image under $\mu$ of $^0\sM(\R)$. Thus, we obtain $M = \;^0M \rtimes A_P$. Since $M$ is $\Theta_K$-stable, we have $K \cap P = K \cap\; ^0M$, and $K \cap P$ is a maximal compact subgroup of $^0M, M$, and $P$.  By definition, the parabolic prk$P$ of $P$ is 
the dimension of $A_P$.

A parabolic pair  is defined to be a pair $(P, A)$ which is given by the group of  real points $P$ of a parabolic subgroup $\sP$ of $\sG$ together with a split component $A$ of $P$. Having fixed a minimal parabolic
$\Q$-subgroup $\sP_0$ of $\sG$,  a parabolic pair $(P, A)$ is called standard if $P_0 \subset P$, $A_P \subset A_0$ where $A_0: = A_{P_0}$. A parabolic pair $(P, A)$ is said to be semi-standard if $A \subset A_0$.

Let $\frak{h}$ be a Cartan subalgebra of $\frak{g}$ which contains $\frak{a}_0$. We denote by $H$  the corresponding Cartan subgroup $Z_G(\frak{h})$ of $G$. Let $\Phi = \Phi(\frak{g}_\C, \frak{h}_\C)$ be the set of roots of $\frak{g}_\C$ with respect to $\frak{h}_\C$, and let $\Phi_{\R} = \Phi_{\R}(\frak{g}_\C, \frak{a}_{0, \C})$ be the set of $\R$-roots, that is, the set of roots of $\frak{g}_\C$ with respect to $\frak{a}_{0, \C}$. Occasionally we may view the elements of $\Phi$ (resp. $\Phi_{\R})$ as roots of $G(\C)$ with respect to $H$ (resp. $A_0$). Given a parabolic pair $(P, A)$, we denote by $\Phi(P, A)$ the set of roots of $P$ with respect to $A$, and by $\Delta(P, A)$ the set of simple roots in $\Phi(P, A)$. Similarly, we will make no disctinction between a character of $A$ and its differential.

 Since we  have fixed a minimal parabolic $\Q$-subgroup $\sP_0$ of $\sG$, together with the unique split component $A_0$ of $P_0$ which is $\Theta_k$-stable, the order on $\Phi_{\R}$ is determined by the condition that the set of positive roots $\Phi^{+}_{\R}$ is equal $\Phi(P_0, A_0)$. We choose an order on $\Phi$ which is compatible  with this order on $\Phi_{\R}$, that is, the restriction of a positive element is positive. Given this order on $\Phi$, we denote by $\Phi^+$ (resp. by $\Delta$) the set of positive roots (resp. simple roots).
 
 Let $(P, A)$ be a semi-standard parabolic pair. Then $\frak{b}: = \frak{h} \cap ^0\frak{m}$ is a Cartan subalgebra of the Lie algebra $^0\frak{m}$ of $^0M$, and one has a direct sum decomposition $\frak{h} = \frak{b} \oplus \frak{a}$. We may (and will) identify $\frak{a}_\C^*$ (resp. $\frak{b}_\C^*$) with the space of linear forms on $\frak{h}$ which vanish on $\frak{b}$ (resp. $\frak{a}$). Thus, there is a canonical isomorphism 
 \begin{equation}
 \frak{h}_\C^* \cong \frak{b}_\C^* \oplus \frak{a}_\C^*.
  \end{equation}
Let $\Phi_M = \Phi(\frak{m}_\C, \frak{h}_\C) = \Phi(^0\frak{m}_\C, \frak{b}_\C)$ denote the set of roots of $\frak{m}_\C$ with respect to $\frak{h}_\C$. Then we can identify $\Phi_M$ with the set of roots which vanish on $\frak{a}$, and $\Delta_M = \Delta \cap \Phi_M$ is the set of simple roots with respect to the order induced by the one on $\Phi$.

Given a parabolic pair $(P, A)$, the element $\rho_P \in \frak{a}^*$ is defined by $\rho_P(a) = (\det \text{Ad}(a)_{\vert \frak{n}})^{1/2}$, $a \in A$,  where $\frak{n}$ denotes the Lie algebra of the unipotent radical $N$ of $P$. As usual, we put 
\begin{equation}
\rho = \frac{1}{2} \sum_{\alpha \in \Phi^+} \alpha \quad \text{and} \quad \rho_{^0M} = \frac{1}{2} \sum_{\alpha \in \Phi^+_M} \alpha.
\end{equation}
Since the parabolic pair $(P, A)$ is semi-standard, we see $\rho_{\vert \frak{a}_0} = \rho_{P_0}$ and $\rho_{\vert \frak{b}} = \rho_{^0M}$. If $(P, A)$ is standard, we have $\rho_{\vert \frak{a}} = \rho_{P}$.

\subsection{Arithmetic quotients and the adjunction of corners}
 Given  a connected semi-simple algebraic $\Q$-group $\sG$ of rk$_\Q \sG > 0$, the group of real valued points $G = \sG(\R)$ is a  real Lie group. Let $K$ be a  maximal compact subgroup of $G$. Since any two of these are conjugate to one another by an inner automorphism, the homogeneous space $G/K = X$ may be viewed as the space of maximal compact subgroups of $G$.  The space $X$ is diffeomorphic to $\R^{d(G)}$, where $\mathrm{d}(G) = \dim G - \dim K$, hence $X$ is contractible. 

An arithmetic subgroup $\Gamma$ of $\sG$ is a discrete subgroup of $G$, and $\Gamma$ acts properly discontinuously on $X$. If $\Gamma$ is torsion-free, the action is free, and the quotient space $\Gamma\setminus X$ is a smooth manifold of dimension $d(G)$. Since the underlying algebraic $\Q$-group $\sG$ has positive $\Q$-rank, $\Gamma\setminus X$ is non-compact but of finite volume. It can be viewed as the interior of a natural compactification with boundary, the adjunction of corners, due to Borel and Serre \cite{BS}. 

This compactification is obtained  as the quotient under $\Gamma$ of  a $\sG(\Q)$-equivariant partial compactification $\overline{X}$. The inclusion $\Gamma \setminus X \longrightarrow \Gamma \setminus \overline{X}$ is a homotopy equivalence.
The boundary $\partial(\Gamma \setminus \overline{X})$ is glued together out of faces $e'(Q)$, where $\sQ$ ranges over a set of representatives for the $\Gamma$-conjugacy classes of proper parabolic $\Q$-subgroups of $\sG$.

 A single face is described by the fibration (induced from the projection $\kappa: \sP \longrightarrow \sP/\sN$)
\begin{equation}\label{fibration}
\Gamma_N \setminus N \longrightarrow e'(P) = \Gamma_P \setminus ^0P / K \cap P \longrightarrow \Gamma_M \setminus Z_M
\end{equation}
over the homogenous  space $\Gamma_M \setminus Z_M$ with compact fibres where $Z_M = (^0P/N)/\kappa (K \cap P)$, $\Gamma_P = \Gamma \cap P,$ $\Gamma_N = \Gamma \cap N$, and $\Gamma_M = \kappa (\Gamma \cap P)$.
In the sequel we will identify the base space $\Gamma_M \setminus Z_M$ with $\mu^{-1}(\Gamma_M)\setminus ^0M /K\cap M$.

\subsection{Cohomology} Let $(\nu, E)$ be a finite-dimensional irreducible representation of the real Lie group $G$ on a real or complex vector space, and let $(\Omega^*(X, E), d)$ be the complex of smooth $E$-valued differential forms on $X$. Given an arithmetic torsion-free subgroup $\Gamma$ of $\sG$, the singular cohomology $H^*(\Gamma \setminus X, \tilde{E})$ of the manifold with coefficients in the local system defined by $(\nu, E)$ is canonically isomorphic to the cohomology $H^*(\Omega(X, E)^{\Gamma})$, the  de Rham cohomology. 

These latter cohomology groups  may be interpreted in terms of relative Lie algebra cohomology groups. Indeed, let $C^{\infty}(G)$ denote the space of $C^{\infty}$-functions on $G$. We endow $C^{\infty}(G) \otimes E$ with the $G$-module structure given as the tensor product of the right regular representation  $l$ of $G$ on $C^{\infty}(G)$ and of $(\nu, E)$. The space $C^{\infty}(G)_K$ of all $C^{\infty}$-vectors $f$ for which $f l(K)$ is a finite-dimensional subspace of $C^{\infty}(G)$ is preserved under the action of the Lie algebra $\frak{g}$, obtained by differentiation of $l$, and compatible  with the action of $K$. Therefore, since $C^{\infty}(G)_K$ is locally finite as a $K$-module, the space of $K$-finite vectors $C^{\infty}(G)_K$ is $(\frak{g}, K)$-module. Taking into account the action of the discrete torsion-free group $\Gamma$, there is an isomorphism of $\Omega^*(G/K, E)^{\Gamma}$ onto the complex
$C^*(\frak{g}, K, C^{\infty}(G)_K \otimes E)$, hence there is an isomorphism
\begin{equation}
 H^*(\Gamma \setminus X, \tilde{E}) \cong H^*(\Omega^*(G/K, E)^{\Gamma}) \tilde{\longrightarrow} H^*(\frak{g}, K, C^{\infty}(G)_K \otimes E).
\end{equation}

Let $L^2(\Gamma\setminus G)$ be the space of square-integrable functions (modulo the centre) on $\Gamma\setminus G$, viewed as usual as a unitary $G$-module via right-translations. The theory of Eisenstein series plays a fundamental role in the description of the spectral decomposition of $L^2(\Gamma \setminus G)$. This space is the direct sum of the discrete spectrum $L^2_{\mathrm{dis}}(\Gamma \setminus G)$, i.e. the span of the irreducible closed $G$-submodules of $L^2(\Gamma \setminus G)$, and the continuous spectrum $L^2_{\mathrm{ct}}(\Gamma \setminus G)$.
The former space contains as a $G$-invariant subspace the space $L^2_{\mathrm{cusp}}(\Gamma \setminus G)$ of cuspidal automorphic forms, i.e. the  cuspidal spectrum. The orthogonal complement in $L^2_{\mathrm{dis}}(\Gamma \setminus G)$ is the residual spectrum, to be denoted $L^2_{\mathrm{res}}(\Gamma \setminus G)$, thus, there is a direct sum decomposition
\begin{equation}
L^2_{\mathrm{dis}}(\Gamma \setminus G) = L^2_{\mathrm{cusp}}(\Gamma \setminus G) \oplus L^2_{\mathrm{res}}(\Gamma \setminus G).
\end{equation} 
The discrete spectrum is a countable Hilbert direct sum of irreducible $G$-modules with finite multiplicities, say\footnote{By abuse of notation, we write $m(\pi, \Gamma) H_{\pi}$ for the direct sum of $m(\pi, \Gamma)$ copies of $H_{\pi}$.}
\begin{equation}\label{disdec}
L^2_{\mathrm{dis}}(\Gamma \setminus G) = \bigoplus_{\pi \in \hat{G}} m(\pi, \Gamma) H_{\pi}.
\end{equation}

By the work of Langlands each of the constituents of the  residual spectrum $L^2_{\mathrm{res}}(G_{\infty}/\Gamma)$ can be structurally described in terms of residues of Eisenstein series attached to irreducible representations occurring in the discrete spectra of the Levi components of proper parabolic subgroups of $G$.

Our object of concern is the cohomology of $\Gamma$ with values in $E$, to be given  in terms of relative Lie algebra cohomology
as
\begin{equation}
H^\ast(\Gamma, E) = H^\ast(\Gamma \setminus X, \tilde{E} ) =  H^\ast(\mathfrak {g}, K; C^{\infty}(\Gamma \setminus G) \otimes_{\C} E)
\end{equation}
 where $C^{\infty}(\Gamma \setminus G) $ denotes the space of $C^{\infty}$-functions on $\Gamma \setminus G$.\footnote{For a differentiable $G$-module $F$ we usually put $H^*(\mathfrak{g}, K, F) = H^*(\mathfrak{g}, K, F_{K})$ where $F_{K}$ denotes the space of $K$-finite vectors in $F$, $K$ a maximal compact subgroup in $G$.}
This cohomology space contains as a natural subspace the square integrable cohomology 
$H^\ast_{(\mathrm{sq})}(\Gamma, E)$,  defined as the image of the homomorphism 
\begin{equation} 
j_{\mathrm{dis}}:  H^{\ast}(\frak{g}, K; L^{2,\infty}_{\mathrm{dis}}(\Gamma \setminus G) \otimes E)\longrightarrow  H^{\ast}(\frak{g}, K; C^{\infty}(\Gamma \setminus G) \otimes E)
\end{equation}
induced in cohomology by the natural inclusion  of the space of $C^{\infty}$-vectors in the discrete spectrum of $\Gamma \setminus G$  into $C^{\infty}(\Gamma \setminus G)$.  In general, the homomorphism $j_{\mathrm{dis}}$ is not injective whereas the homomorphism induced by the  inclusion of the space of $C^{\infty}$-vectors in the  cuspidal spectrum into $C^{\infty}(\Gamma \setminus G)$ is injective; its image is called the cuspidal cohomology of $\Gamma$, to be denoted $H^\ast_{\mathrm{cusp}}(\Gamma \setminus X, E)$

Given a coefficient system $E$, there are, up to infinitesimal equivalence,  only finitely many irreducible unitary representations $(\pi, H_{\pi})$ with non-zero cohomology 
$H^\ast(\mathfrak {g}, K; H_{\pi} \otimes_{\C} E)$. Thus, decomposition \ref{disdec} gives rise to a decomposition of the cuspidal cohomology
\begin{equation}
H^\ast_{(\mathrm{cusp})}(\Gamma \setminus X, E) \cong \bigoplus_{\pi \in \hat{G}} m(\pi, \Gamma) H^\ast(\mathfrak {g}, K; H_{\pi} \otimes_{\C} E)
\end{equation}
as a finite algebraic direct sum.

\subsection{Cuspidal cohomology for $GL_m/\Q$ and $SL_m/\Q$}\label{cusp} The following result (cf.\cite[Theorem 3.3]{S3}) is fundamental for a better comprehension of the cuspidal cohomology.
\begin{proposition}Let $\Gamma$ be an arithmetic subgroup of $GL_m/\Q$, and fix $(\nu, E)$. Up to infinitesimal equivalence, there is at most one irreducible unitary representation 
$(\sigma, H_{\sigma})$ of $GL_m(\R)$ such that the relative Lie algebra cohomology $H^*(\frak{gl}_m(\R), O(m); H_{\sigma}^{\infty} \otimes E) \neq \{0\}$, and $(\sigma, H_{\sigma})$ occurs with non-zero multiplicity in the cuspidal spectrum of $\Gamma \setminus GL_m(\R)$. Such a representation is necessarily tempered.
\end{proposition}

As a tempered representation $(\sigma, H_{\sigma})$ has vanishing Lie algebra cohomology outside a certain range centered around the middle dimension, see \cite[III, 5.3.]{BW}.The bounds of this range, depending only on $m$, are determined in \cite[Proposition 3.5.]{S3}. We only need the upper bound which is given as follows:
\begin{equation}
v_o(m): = \frac{(m + 1)^2 - 1}{4} - 1 \; ,\text{if}\; m \; \text{even, and}  \quad v_o(m): =  \frac{(m + 1)^2}{4} - 1\; ,\text{if}\; m \; \text{odd}.
\end{equation}
We observe that for $m > 3$ the cohomological dimension $ \frac{m(m-1)}{2}$ of an arithmetic group in $GL_m(\Q)$ is always larger than $v_o(m)$. For $m = 2, 3$, we have equality.

An analogous results holds in the case of the group $SL_n^{\pm}(\R)$, and, with a slight modification regarding the number of irreducible unitary representations which are cuspidal and cohomological, for the group $SL_n$ (see  \cite[Proposition 3.5.]{S3}).

\begin{remark} When $m \geq 3$, the group $SL_m$ has the congruence subgroup property. Therefore, we can infer from the adelic result of Franke and the preceding discussion, that, for $m > 3$,  there is no cuspidal cohomology in degree $\mathrm{cd}(\Gamma) = \frac{m(m-1)}{2}$, thus, the total cohomology in this degree comes from the faces in the Borel-Serre compactification. Which ones actually contribute is 
determined later on.
\end{remark}
 
\begin{example}We consider the case of the group $GL_2(\R)$, that is, $m = 2$.
Let $V(r), r \geq 2,$ denote the irreducible two-dimensional representation of the orthogonal group $O(2)$ which is fully induced by the character $k_{\theta} \mapsto e^{ir\theta}$ of the subgroup $SO(2)$ of rotations $k_{\theta}, \theta \in [0, 2\pi]$, in $O(2)$ of index two.
Given an integer $\ell \geq 2$, we denote by  $D_{\ell}$ the discrete series representation of $GL(2,\R)$ of lowest $O(2)$-type $V(\ell)$. The  representation $D_{\ell}$ is  square-integrable and characterized by the fact that its restriction to the maximal compact subgroup $O(2)$ of $GL(2,\R)$ decomposes as an algebraic sum of the form $$D_{\ell} \vert _{O(2)} \cong \oplus_{r \in \Sigma(\ell)} V(r), \quad \quad \Sigma(\ell) = \{l \in \Z \; \vert \; l \equiv \ell \bmod 2, l \geq \ell \}.$$

In this labelling of the discrete series representations of $GL(2,\R)$ the Harish-Chandra parameter of $D_{\ell}, \ell \geq 2,$ is $\ell-1$.

Let $(\sigma_k, F_{k}), k \geq 0,$ be  the irreducible finite-dimensional representation of $GL(2, \R)$ of highest weight $\mu_k = k \cdot \omega$ [where $\omega$ denotes the fundamental dominant weight of $GL(2, \R)$],  thus, $\dim F_k = k + 1$.  The  cohomology 
$
H^\ast (\frak{gl}_2(\R), O(2);  D_{\ell} \otimes F_k)$ vanishes if $ k \neq \ell-2$ since the infinitesimal character $\chi_{D_{\ell}}$ differs from   the one of the contragredient representation of $(\sigma_k, F_k)$. In the case  $k = \ell-2$ one has  $H^{q} (\frak{gl}_2(\R), O(2);D_{\ell} \otimes F_{\ell-2}) = \C$   for $q = 1$; it vanishes  otherwise.
\end{example}

\subsection{Cohomology of a face $e'(P)$ in the boundary of $\Gamma \setminus \overline{X}$}

Let  $e'(P)$, $\sP \in \mathcal{P}$, $\sP \neq \sG$ be a single face in the boundary $\partial(\Gamma \setminus \overline{X})$. The spectral sequence in cohomology associated to the fibration \ref{fibration} of $e'(P)$ degenerates at $E_2$ \cite[Theorem 2. 7.]{S2}.  The deRham cohomology spaces of the fiber $\Gamma_N \setminus N$ can be identified with the Lie algebra cohomology $H^*(\frak{n}, E)$ of the Lie algebra $\frak{n}$ of $N$. This latter cohomology space carries a natural structure as an $M$-module; its restriction to $\Gamma_M$ coincides with the natural action of $\Gamma_M$ on the cohomology of the fiber. This gives rise to the identification 
\begin{equation}
H^*(e'(P), E) \cong H^*(\Gamma_M \setminus Z_M, H^*(\frak{n}, E)).
\end{equation}
This cohomology space contains the cuspidal cohomology of the face as a natural subspace. Using the  interpretation  in terms of relative Lie algebra cohomology, it is defined to be
\begin{equation}\label{deccusp}
H^*_{\mathrm{cusp}}(\Gamma_M \setminus Z_M, H^*(\frak{n}, E)): = \bigoplus_{(\pi, H_{\pi})} H^*(^0\frak{m}, K_M; V_{\pi} \otimes H^*(\frak{n}, E))
\end{equation}
where $V_{\pi}$ denotes the $\pi$-isotypic component of the space of cusp forms for $\Gamma_M \backslash Z_M$  for a given irreducible unitary representation $(\pi, H_{\pi})$ in the unitary dual of $^0M$.

Following Kostant (\cite{Ko}, Theorem.~5.13), the Lie algebra
cohomology $H^\ast(\mathfrak{n},E)$ of $\mathfrak{n}$ with
coefficients in the irreducible representation $(\nu ,E)$ of
$G$ is given as a
$M$--module as
the sum
\begin{equation}\label{Kostant}
H^\ast(\mathfrak{n},E)=\bigoplus_{w\in W^{P}} F_{\mu_w}
\end{equation}
where the sum ranges over $w$ in the set $W^{P}$ of the minimal
coset representatives for the left cosets of the Weyl group $W = W(\frak{g}_\C, \frak{h}_\C)$ modulo the Weyl
group $W_{P} = W(\frak{m}_\C, \frak{h}_\C) = W(^0\frak{m}_\C, \frak{b}_\C)$
of the Levi factor $M$ of $P$. The letter $F_{\mu_w}$
denotes the irreducible finite-dimensional
$M(\C)$-module of highest weight
\begin{equation}
\label{eq:HighestWghtMu}
\mu_w=w(\Lambda +\rho_{P_0})-\rho_{P_0},
\end{equation}
where $\Lambda\in\check{\mathfrak{a}}_{P_0,\mathbb{C}}$ is the
highest weight of $(\nu,E)$. The weights $\mu_w$ are all dominant
and distinct and, given a fixed degree $q$, only the weights
$\mu_w$ with length $\ell(w)=q$ occur in the decomposition of
$H^q(\mathfrak{n},E)$ into irreducibles. As usual, the length of $w \in W$ is meant with respect to the set of simple reflections $w_{\alpha} \in W, \alpha \in \Delta$. The set $W^P$ can also be described as the set $\{w \in W\; \vert \; w^{-1}(\Delta_M) \subset \Phi^{+}$\}.

Since $M$ is the direct product $A_P \times\;^0M$ of the split component $A_P$ and $^0M$, the module $F_{\mu_w}$, viewed as $^0M$-module, is irreducible. With respect to the $A_P$-action, we get the analogue of the decomposition \ref{Kostant} where  $F_{\mu_w}$ is viewed as an $A_P$-module of highest weight ${\mu_w}_{\vert \frak{a}_P}$

\begin{definition}
Let $[\phi]$ be a non-trivial cohomology class in $H^*(\Gamma_M \setminus Z_M, H^*(\frak{n}, E))$ represented by a differential form $\phi \in \Omega^*(\Gamma_M \setminus Z_M, H^*(\frak{n}, E))$. We say that $[\phi]$ is a class of type $\mu \in \frak{a}^*_P$ if $[\phi] \in  H^*(\Gamma_M \setminus Z_M, F_{\mu})$
\end{definition}

Merging this decomposition of $H^*(\frak{n}, E)$ as an $^0M$-module and the decomposition  \ref{deccusp} of the cuspidal cohomology $H^*_{\mathrm{cusp}}(e'(P), E)$ of the face $e'(P)$ we obtain the following decomposition of the latter as a finite algebraic direct sum

\begin{equation}\label{deccuspcoh}
H^*_{\mathrm{cusp}}(\Gamma_M \setminus Z_M, H^*(\frak{n}, E)) = \bigoplus_{(\pi, H_{\pi})} \bigoplus_{w \in W^P} H^*(^0\frak{m}, K_M; V_{\pi} \otimes F_{\mu_w})
\end{equation}

As in \cite[Section
3.2]{S2}, this decomposition gives rise to the following essential concept.
\begin{definition}Let $[\phi]$ be a non-trivial cohomology class in $H^*_{\mathrm{cusp}}(\Gamma_M \setminus Z_M, H^*(\frak{n}, E))$ represented by a differential form $\phi \in \Omega^*(\Gamma_M \setminus Z_M, H^*(\frak{n}, E))$. We say that $[\phi]$ is a class of type $(\pi, w)$ if there exist an irreducible unitary representation $\pi$ of $^0 M$ which occurs in the cuspidal spectrum of $\Gamma_M \setminus ^0 M$ and $w \in W^P$ such that via the identification \ref{deccuspcoh} $[\phi] \in 
H^\ast \left(^0\frak{m}, K_M;
V_{\pi} \otimes F_{\mu_w}\right).$
\end{definition}

If the infinitesimal
character $\chi_{\pi}$ 
of $\pi$ does not coincide with the infinitesimal
character of the representation contragredient to $F_{\mu_w}$, the
cohomology space $H^\ast \left(^0\frak{m}, K_M;
V_{\pi}\otimes F_{\mu_w}\right)$ vanishes, that is, there
are no classes of type $(\pi,w)$.
Moreover, if the module $F_{\mu_w}$ is not isomorphic to its complex
conjugate contragredient $\overline{F}_{\mu_w}^\ast$, the Lie algebra cohomology $H^\ast
\left(^0\frak{m}, K_M; V_{\pi}\otimes
F_{\mu_w}\right)$ vanishes, since this condition implies that the
complex contragredient of $F_{\mu_w}$ and $V_{\pi}$ have distinct
infinitesimal character.

\section{Eisenstein cohomology classes}\label{eisen} We review how the theory of Eisenstein series can be used to construct certain cohomology classes in $H^*(\Gamma \setminus X; E)$ which are represented by a regular value of a suitable Eisenstein series attached to  cuspidal cohomology classes in $H^*_{\mathrm{cusp}}(e'(P), E)$. In this Section and the subsequent ones  we have to assume some familiarity  with the theory of Eisenstein series as given in Langlands \cite{La1} and Harish-Chandra \cite{HC} and the general results regarding the construction carried through in \cite{S1}, \cite{S2}, and \cite{S4}.
\subsection{The construction}
Let  $[\phi] \in  H^*_{\mathrm{cusp}}(e'(P), E) = H^*_{\mathrm{cusp}}(\Gamma_M \setminus Z_M, H^*(\frak{n}, E))$ be a non-trivial cohomology class 
of type $(\pi, w)$, 
$(\pi \in\; ^0 \hat{M}, w \in W^P)$, represented by a differential form $\phi \in \Omega^*(\Gamma_M \setminus Z_M, H^*(\frak{n}, E))$. Following  \cite[Chapter 4]{S2}, 
we associate to $\phi$ via the differential form $\phi_{\lambda} = \phi a^{\lambda + \rho}$ in
 $\Omega^*(\Gamma_P \setminus X, E) $ the Eisenstein series 
 \begin{equation}
 E(\phi, \lambda): = \sum_{\Gamma_P \setminus \Gamma} \gamma \circ \phi_{\lambda}.
  \end{equation}
  
  This Eisenstein series is first defined for all $\lambda$ in 
  $(\frak{a}^*_{\C})^+ = \{ \lambda \in \frak{a}^*_{\C}\; \vert \; \text{Re} \lambda \in \rho_P + (\frak{a}^*)^+\}$
  where $(\frak{a}^*)^+ = \{ \lambda \in \frak{a}^*\; \vert \; (\lambda, \alpha) > 0 \; \text{for all}\; \alpha \in \Delta(P, A)\}$
  and is holomorphic in that region. Via analytic continuation it admits a meromorphic continuation to all of $\frak{a}^*_{\C}$. 
  We fix $\lambda_0 \in \frak{a}^*_{\C}$. If the Eisenstein series $E(\phi, \lambda)$ is holomorphic at this point, then 
  $E(\phi, \lambda_0)$ is an $E$-valued, $\Gamma$-invariant differential form on $X$. The following result 
  \cite[Theorem 4.11.]{S2} is decisive for the 
  construction of Eisenstein cohomology classes.
  
  \begin{theorem}\label{conEis} Let $\sP$ be a parabolic $\Q$-subgroup of $\sG$, and let $A_P$ be the uniquely determined  
  split component of $P = \sP(\R)$ which is stable under the Cartan involution $\Theta_K$. Let 
  $[\phi]$ be a non-trivial cohomology class in $H^*_{\mathrm{cusp}}(e'(P), E) = H^*_{\mathrm{cusp}}(\Gamma_M \setminus Z_M, H^*(\frak{n}, E))$ 
of type $(\pi, w)$, with 
$\pi$ an irreducible unitary representation occurring in the cuspidal spectrum of $\Gamma_M \setminus ^0M$ and 
 $w \in W^P$, represented by a differential form $\phi \in \Omega^*(\Gamma_M \setminus Z_M, H^*(\frak{n}, E))$.
 If the Eisenstein series $E(\phi, \lambda)$ assigned to $[\phi]$ is holomorphic at the point 
 \begin{equation}
 \lambda_0 = - w(\Lambda + \rho)_{\vert \frak{a}_P}
  \end{equation} $($which is real and uniquely determined by $[\phi]$$)$  then $E(\phi, \lambda_0)$
  is a closed harmonic differential form on $\Gamma \setminus X$ and represents a non-trivial cohomology class $[E(\phi, \lambda_0)]$
  in $H^*(\Gamma \setminus X ; E)$. We call such a cohomology class a regular Eisenstein cohomology class.
  \end{theorem}
 
 \subsection{Restriction maps and the constant term}
 Given a parabolic $\Q$-subgroup $\sQ$ of $\sG$, the image of a cohomology class $[E(\phi, \lambda_0)]$ under the restriction map
  $r^*_Q: H^*(\Gamma \setminus \overline{X}, E) \longrightarrow H^*(e'(Q), E)$ is given by $[E(\phi, \lambda_0)_{Q}]_{\vert e'(Q)}$, that is, equals the restriction to $e'(Q)$ of the constant Fourier coefficient $E(\phi, \lambda_0)_Q \in \Omega^*(\Gamma_Q \setminus X, E)$ of $E(\phi, \lambda_0)$ along the parabolic $Q$ (see \cite[Proposition 1. 10.]{S2}). This result is already used in proving the non-vanishing of the class $[E(\phi, \lambda_0)]$. In fact, one shows $[E(\phi, \lambda_0)_{P}]_{\vert e'(P)} = [\phi] \neq 0.$
  
 For later use, we have to recall the following case where the notion of  associate parabolic subgroups plays a role.
  Let $\sP, \sQ$ be two parabolic $\Q$-subgroups of $\sG$, and let $A_P, A_Q$ be the corresponding $\Theta_K$-stable split components. We denote by $W(A_P, A_Q)$ the set of isomorphisms $A_P \tilde{\longrightarrow} A_Q$ which are induced by those inner automorphism of $\sG(\Q)$ defining  a $\Q$-rational  isomorphism $M_P \tilde{\longrightarrow} M_Q$. Equivalently, the split components $A_P$ and $A_Q$ are conjugate in $G$ under an element in $\sG(\Q)$. We say that $\sP$ and $\sQ$ are associate if $W(A_P, A_Q) \neq \emptyset$. The notion of being associated defines an equivalence relation on the set of parabolic $\Q$-subgroups $\sG$. Let $\mathcal{C}$ be the set of classes of associate parabolic $\Q$-subgroups of $\sG$. Note that the minimal parabolic $\Q$-subgroups form one class $\{\sP_0\} \in \mathcal{C}$, represented by the standard minimal parabolic $\Q$-subgroup $\sP_0$.
   
  Now suppose that the  parabolic $\Q$-subgroups $\sP$ and $\sQ$ are in the same associate class. Then we have
   \begin{equation}
  r^*_Q([E(\phi, \lambda_0)]) = \sum_{s \in W(A_P, A_Q)} [c(s, \lambda_0)_{s\lambda_0}(\phi_{\lambda_0})]_{\vert e'(Q)}
   \end{equation}
   where $c(s, \lambda_0)_{s\lambda_0}: \Omega^*(\Gamma_P \setminus X, E) \longrightarrow \Omega^*(\Gamma_Q \setminus X, E)$ is a certain
   ``intertwining" operator defined in \cite[4. 10]{S2}. 
The cohomology class $[c(s, \lambda_0)_{s\lambda_0}(\phi_{\lambda_0})]_{\vert e'(Q)}$ in $H^*_{\mathrm{cusp}}(e'(Q), E)$ is a class of weight 
$v_s(\Lambda + \rho) - \rho_{\vert \frak{a}_Q}$ where $v_s \in  W^Q$ is a uniquely determined element with
\begin{equation}
v_s(\Lambda + \rho) - \rho_{\vert \frak{a}_Q} + s\lambda_0 = 0 \quad \text{and}\quad  \chi_{^s\pi} = \chi_{- v_s(\Lambda + \rho)\vert_{\frak{b}_{Q, \C}}}.
\end{equation}
The latter equation refers to an equality  for the infinitesimal character of the representation $^s\pi$ of $^0M_Q$ obtained from the representation $\pi$
 of $^0M_P$ under the twist induced by $s \in W(A_P, A_Q)$.

\section{The case $SL_n/\Q$ - the faces which matter}\label{paramatter}

Now we consider the $\mathbb{Q}$-split
simple simply connected special linear $\Q$-group $\sG = SL_n$ of
$\mathbb{Q}$-rank $n-1$, where $n\geq 2$. We fix the maximal compact subgroup $K = SO(n)$ in the real Lie group $G = SL_n(\R)$. The symmetric space $X$ associated to $G$ has dimension $\mathrm{d}(SL_n) = (1/2)(n-1)(n+2)$. Given an arithmetic torsion-free subgroup $\Gamma$ of $SL_n(\Q)$, in the following sections, we study the construction of Eisenstein cohomology classes in the cohomology $H^*(\Gamma \setminus X; E)$ in the degree given by the cohomological dimension $\mathrm{cd}(\Gamma)  = \frac{n(n-1)}{2}$. By the very definition of this notion, we have $H^r(\Gamma \setminus X; E) = \{0\}$ for $r > \mathrm{cd}(\Gamma)$.

In this Section, we determine which classes $\{\sP\}$ of associate parabolic $\Q$-subgroups of  $\sG = SL_n$ may contribute to the Eisenstein cohomology in the decomposition \ref{autodecomp} in Section \ref{intro}.

\subsection{}
Let $\sP_0$ be in $SL_n/\Q$ the  minimal
parabolic $\mathbb{Q}$-subgroup of upper triangular matrices, and $\sP_0 = \sL_0 \sN_0$ its Levi
decomposition where $\sN_0$ is the unipotent radical. The maximal
split torus $\sL_0$ is given as the group $\{\diag(t_1, \ldots, t_n) =:t \in SL_n(\Q)\}$ of diagonal matrices with determinant one.

 Let
$\Phi_\Q$, $\Phi^+_\Q$, $\Delta_\Q$ denote the corresponding sets of roots,
positive roots, simple roots, respectively. We may (and will) identify $\Phi_\Q$ and $\Phi_\R$ resp. $\Phi$. 
If $e_i$ is the
projection of $\sL_0$ to its $i^{\mathrm{th}}$ component, then
$
\Delta_\Q =\{\alpha_i=e_i\cdot e^{-1}_{i+1} \; \vert \; i = 1, \ldots, n-1\}$.

The conjugacy classes of parabolic $\mathbb{Q}$-subgroups of $SL_n$ are in
bijection with the subsets of the set $\Delta_{\Q}$ of simple roots. Each of the following standard parabolic $\Q$-subgroups represents a $SL_n(\Q)$-conjugacy class.
Given $J \subset \Delta_\Q$, let $\sS_J=\left(\cap_{\alpha\in J}{\Ker}(\alpha)\right)^\circ$ be the corresponding subtorus in $\sL_0$.
We denote by $\sP_J$ the
corresponding parabolic $\mathbb{Q}$-subgroup containing $\sP_0$, defined to be
$\sP_J= \sL_J \sN_J$, where
$\sL_J$ is the Levi factor, i.e.~the centralizer of
$\sS_J$,
and $\sN_J$ is the unipotent radical. The characters of $\sL_0$ in the Lie algebra of $\sN_j$ are the positive roots which contain at least one simple root outside of $J$. The roots of $\sL_J$ are those roots whose simple components are in $J$. Evidently, $\sP_{\emptyset} = \sP_0$ and $\sP_{\Delta_\Q} = SL_n$, and, given two subsets $I, J \subset \Delta$, where $I \subset J$, then $\sP_I \subset \sP_J$. Moreover, $\sP_I \cap \sP_J = \sP_{I \cap J}$.

\subsection{A choice of faces in the adjunction of corners}
We are interested to construct non-vanishing classes in the degree given by the cohomological dimension $\mathrm{cd}(\Gamma)$ of $\Gamma$, that is, 
in the highest possible degree in which cohomology  may exist at all. Thus, being interested in Eisenstein cohomology classes, we have to determine for which proper  parabolic subgroups $\sP$, up to $\Gamma$-conjugacy, 
the corresponding  cuspidal cohomology groups $H^{\mathrm{cd}(\Gamma)}_{\mathrm{cusp}}(e'(P), E)$ are non-zero in this degree. However, in most cases, these groups are zero.
This is a consequence of the vanishing result \ref{cusp} concerning the cuspidal cohomology for arithmetic subgroups in some $GL_m/\Q$ in degrees outside a certain range centered around the middle dimension.

The following proper parabolic $\Q$-subgroups will be at our disposal for the envisaged construction:

First, say case $\mathrm{(I)}$,  a  minimal parabolic $\Q$-subgroup $\sB$ (also called Borel subgroup) is $\sG(\Q)$-conjugate to the standard minimal parabolic $\Q$-subgroup $\sP_0 = P_{\emptyset}$. 
The corresponding set $W^{B}$ of minimal coset representatives coincides with the Weyl group $W$, and the longest element $w_B$ in $W$ has length $\ell(w_B) = \dim N_0 = \frac{n(n-1)}{2} = \mathrm{cd}(\Gamma)$. Thus, we have $F_{\mu_{w_B}} = H^{\mathrm{cd}(\Gamma)}(\frak{n}_0, E)$.

Second, labelled case $\mathrm{(II)}$, denote by $\mathcal{J}^{\mathrm{cd}}$ the family of non-empty sets
 $J \subset \Delta_{\Q}$ subject to the condition that if $\alpha_i, \alpha_{i+1} \in J$, with $i \in \{1, \ldots, n-2\}$, then $\alpha_{i+2} \notin J$. The corresponding standard parabolic subgroup $\sP_J$ presents itself as the stabilizer within $SL_n$ of the flag $V_{i_1} \subsetneq V_{i_2} \subsetneq \ldots \subsetneq V_{i_r}$ where the tuple $(i_1, i_2, \ldots, i_r)$ is given by the set of indices in ascending order of the simple roots in the complement $CJ$ of $J$ in $\Delta_{\Q}$. Given $1 \leq s \leq n-1$, $V_s$ is defined to be $\Q f_1 + \Q f_2 + \ldots + \Q f_s$ with regard to the standard basis of $\Q^n$. This condition assures that only blocks of size one, two or three make up the Levi subgroup of $\sP_J$. In the latter two cases,  the upper bound $v_o(m)$ of the range in which the group $GL_m$ has possible non-vanishing cuspidal cohomology classes coincides with  $\frac{m(m-1)}{2}$ when $m = 2, 3$. Consequently, adding up this upper bound over all blocks of size two or three, we get  
\begin{equation}
\sum_{\text{2-blocks}} v_o(2) + \sum_{\text{3-blocks}} v_o(3) + \dim N_J = \mathrm{cd}(\Gamma).
\end{equation}
By definition, the case $\mathrm{(II)}$ comprises all parabolic $\Q$-subgroups which are $\sG(\Q)$-conjugate to a standard parabolic subgroup $\sP_J$ with $J \in \mathcal{J}^{\mathrm{cd}}$.
The cases $\mathrm{(I)}$ and $\mathrm{(II)}$ of proper parabolic $\Q$-subgroups exhaust, up to $\Gamma$-conjugacy, all possible sources  of Eisenstein cohomology classes of degree $\mathrm{cd}(\Gamma)$. This follows from the fact that, if a block of size $m \geq 4$ occurs as a factor in the Levi subgroup of a parabolic $\Q$-subgroup $\sQ$, the bound $v_o(m)$ is smaller than  $\frac{m(m-1)}{2}$, thus, the cuspidal cohomology group  $H^{\mathrm{cd}(\Gamma)}_{\mathrm{cusp}}(e'(Q), E)$ of the face $e'(Q)$ is zero.
Therefore we obtain

\begin{proposition}\label{propparamatter} Let $\{P\}$ be an associate class of proper parabolic $\Q$-subgroups of $SL_n/\Q$. If $\sP$ is neither a minimal parabolic $\Q$-subgroup nor $\sP$ is $\sG(\Q)$-conjugate to a standard parabolic subgroup $\sP_J$ where the defining set  $J \subset \Delta_{\Q}$ is an element of $\mathcal{J}^{\mathrm{cd}}$ then the summand $H^{\mathrm{cd}(\Gamma)}_{\{\sP\}}(\Gamma \setminus X, E)$ in the decomposition (\ref{autodecomp}) vanishes in degree $\mathrm{cd}(\Gamma)$.
\end{proposition}

\section{Construction of Eisenstein cohomology classes - the  case $(\mathrm{II})$}\label{ParaII}
In this section we carry through the construction of non-vanishing cohomology classes in the top cohomology $H^{\mathrm{cd}(\Gamma)}(\Gamma \setminus X; E)$ which originate from faces $e'(Q)$ where the associate class of the corresponding parabolic $\Q$-subgroup falls into case $\mathrm{(II)}$. The main result is Theorem \ref{main}.
\subsection{Block-2 parabolic subgroups of $\sG$} For the sake of clarity we first deal with the following specific family of  parabolic $\Q$-subgroups belonging to case $\mathrm{(II)}$.
Let $J = \{\alpha_{i_1}. \ldots, \alpha_{i_r} \} \subset \Delta_{\Q}$ be a non-empty set of simple roots of $SL_n/\Q$ 
subject to the conditions that if $\alpha_i, \alpha_j \in J$ then $\vert i - j \vert \geq 2.$, and let $\sP_J$ be the corresponding parabolic $\Q$-subgroup of $\sG/\Q = SL_n/\Q$. The parabolic rank of $\sP_J$ is equal $(n-1) - r$. We observe that the blocks in the Levi subgroup of $\sP_J$ can only have size one or two, and at least one block of size two occurs. Therefore we call a parabolic $\Q$-subgroup $\sP_J$ where $J$ satisfies this condition a block-$2$ parabolic $\Q$-subgroup of $\sG$.

We give two examples: First, we suppose that $n = 2m$ is even, and $n \geq 4$. We take the set $J = \{\alpha_1, \alpha_3, \ldots, \alpha_{n-3}, \alpha_{n-1}\}.$ Then the Levi subgroup $\sL_J$ of the corresponding parabolic $\Q$-subgroup $\sP_J$ is given via ``diagonal matrices" as
\begin{equation}
\sL_J = \{\textrm{diag}(A_1, A_3, \dots, A_{n-3},  A_{n-1}) \in SL_n(\Q)\; \vert \; A_i \in GL_2(\Q)\}. 
\end{equation}
The parabolic rank of $\sP_J$ is $\frac{n}{2}$.
Second, let $n \geq 3$, and take $J' = \{\alpha_j\}$ for any simple root in $\Delta_{\Q}$. The Levi subgroup of the corresponding parabolic $\Q$-subgroup $\sP_{J'}$
is given as
\begin{equation}
\sL_{J'} = \{\textrm{diag}(a_1, \ldots, a_{j-1}, A, a_{j+2}, \ldots, a_n) \in SL_n(\Q) \; \vert \;  A \in GL_2(\Q), a_i \in GL_1(\Q) \}.
\end{equation}

\begin{theorem}\label{Eisconstruct}Given $n \geq 3$, let $\sP = \sP_J$ be the parabolic $\Q$-subgroup of $SL_n/\Q$ defined by the non-empty set $J = \{\alpha_{i_1}. \ldots, \alpha_{i_r} \} \subset \Delta_{\Q}$ subject to the conditions that if $\alpha_i, \alpha_j \in J$ then $\vert i - j \vert \geq 2.$ Given an arithmetic torsion-free subgroup $\Gamma \subset SL_n(\Q)$, and a rational finite-dimensional representation $(\nu, E)$ of $SL_n(\R)$ of highest weight $\Lambda$, let $[\phi] \in H^{\mathrm{cd}(\Gamma)}_{\mathrm{cusp}}(e'(P), E)$, $[\phi] \neq 0$, be a cuspidal cohomology class of type $(\pi, w_{P_J})$ where $\pi$ is an irreducible unitary representation of $^0 M_J$ and $w_{P_J}$ denotes the longest element in $W^{P_J}$. Then the Eisenstein series $E(\phi, \lambda)$  attached to the differential form $\phi$ is holomorphic in $\lambda_0 = - w_{P_J}(\Lambda + \rho)_{\vert \frak{a}_J}$, and $[E(\phi, \lambda_0)]$ is a non-zero class in $H^{\mathrm{cd}(\Gamma)}(\Gamma \setminus X, E)$ which is represented by the closed, harmonic differental form $E(\phi, \lambda_0)$.
\end{theorem}

\begin{proof} As stated in \ref{conEis}, the point in question  $\lambda_0 = - w_{P_J}(\Lambda + \rho)_{\vert \frak{a}_{P_J}}$ is real valued. Let $w_G$ be the longest element in $W$, and $w_M$ the longest element in $W_{P_J}$; we have $\ell(w_G) = \vert \Phi^+ \vert$, and $\ell(w_M) = \vert \Phi_M^+ \vert$. The product $w_{P_J} = w_M\cdot w_G$ has length $\ell(w_G) - \ell(w_M) = \dim N_P$, and $w_{P_J}$ is the longest element in the set $W^P$ of minimal coset representatives for the cosets $W_M \setminus W_G$. Since $W_M$ acts trivially on $\frak{a}_{P_J}$ we obtain  $\lambda_0 = - w_G(\Lambda + \rho)_{\vert \frak{a}_{P_J}}$. The highest weight $\Lambda$ of $(\nu, E)$ is transferred under $w_G$ into the highest weight $\tilde{\Lambda}$ of the representation which is contragredient to $(\nu, E)$. In particular, we have $- w_G(\rho) = \rho$. Therefore we obtain
\begin{equation}
\lambda_0 = \tilde{\Lambda} + \rho_ {\vert \frak{a}_{P_J}}.
\end{equation}
The highest weight $\tilde{\Lambda}$ is dominant because $\Lambda$ is dominant. Thus we have the estimate
\begin{equation}\label{Wallach}
(\lambda_0, \alpha) = (\tilde{\Lambda}_{\vert \frak{a}_{P_J}}, \alpha) + (\rho_{\vert \frak{a}_{P_J}}, \alpha) \geq (\rho_{P_J}, \alpha) \; \text{for all}\;  \alpha \in \Delta(P_J, A_{P_J}).
\end{equation}

For the sake of simplicity, since the parabolic group $\sP_J$ is fixed, we now write $\frak{a}$ instead of $\frak{a}_{P_J}$. The region of absolute convergence of the Eisenstein series $E(\phi, \lambda)$  is given as 
 \begin{equation}
 (\frak{a}^*_{\C})^+ = \{ \lambda \in \frak{a}^*_{\C}\; \vert \; \text{Re} \lambda \in \rho_P + (\frak{a}^*)^+\}
 \end{equation}
  where $(\frak{a}^*)^+ = \{ \lambda \in \frak{a}^*\; \vert \; (\lambda, \alpha) > 0 \; \text{for all}\; \alpha \in \Delta(P, A)\}$.
 The Eisenstein series is holomorphic there and admits a meromorphic continuation to all of $\frak{a}^*_{\C}$. 

Suppose now that the Eisenstein series $E(\phi, \lambda)$ has a pole at the point $\lambda_0$. Then, given the data $(P_J, \pi, \lambda_0)$ the corresponding Langlands quotient
 $J_{P_J, \pi, \lambda_0}$ would carry as a $(\frak{sl}_n(\R), K_{\R})$-module a unitary structure. Note that $\sP_J \neq \sP_0$ and that the tempered representation $\pi$ is not the trivial representation. Therefore, using a criterion, due to Wallach (see \cite[Chap. IV, \S 5]{BW}), this implies $\rho_{P_J}(a) > \text{Re} \lambda_0(a)$ for all $a \in C\ell(\frak{a}^+) = \{a \in \frak{a}\; \vert \; \beta(a) \geq 0, \beta \in \Delta(P_J, A_{P_J})\}$. However, this contradicts equation \ref{Wallach}. Therefore, the Eisenstein series $E(\phi, \lambda)$ is holomorphic at the point $\lambda_0 = - w_{P_J}(\Lambda + \rho)_{\vert \frak{a}_{P_J}}$. Then Theorem \ref{conEis} implies the claim.
\end{proof}

We define $H^{\mathrm{cd}(\Gamma)}(\Gamma \setminus X; E)_{e'(P)}$ to be the subspace of $H^{\mathrm{cd}(\Gamma)}(\Gamma \setminus X, E)$ which is generated by all
 non-zero Eisenstein cohomology classes $[E(\phi, \lambda_0)]$ where $[\phi]$ ranges over all non-zero cuspidal classes in $H^{\mathrm{cd}(\Gamma) - \ell(w_P)}(\Gamma_M \setminus X_M, F_{\mu_{w_P}})$. We are interested in the relation of this space
 to analogous  spaces attached to a different parabolic $\Q$-subgroup which is not $\Gamma$-conjugate to $\sP$.
The theory of Eisenstein series, in particular, the properties of the constant terms of a given Eisenstein series, require 
to  group together the various contributions originating from faces $e'(P)$, where, up to $\Gamma$-conjugacy, $\sP$ ranges over the finitely many elements in a given associate class $\{\sQ\} \in \mathcal{C}$ of parabolic subgroups. 

We retain the notation and assumptions of Theorem \ref{Eisconstruct}. Suppose that $[E(\phi, \lambda_0)]$ is a non-zero class in $H^{\mathrm{cd}(\Gamma)}(\Gamma \setminus X, E)$ which is represented by the closed, harmonic differental form $E(\phi, \lambda_0)$ and which originates with a non-zero   cuspidal cohomology class  $[\phi]$ in $H^*_{\mathrm{cusp}}(e'(P), E)$
of type $(\pi, w_{P_J})$ where $\pi$ is an irreducible unitary representation of $^0 M_J$ and $w_{P_J}$ denotes the longest element in $W^{P_J}$. Let $\sQ$ be a proper parabolic $\Q$-subgroup of $\sG$. If  $\textrm{prk}(Q)  > \textrm{prk}(P_J)$ or if $\textrm{prk}(Q)  = \textrm{prk}(P_J)$ and $\sQ$ and $\sP_J$ are not associate to one another then the constant Fourier coefficient  $E(\phi, \lambda)_Q$ along $Q$ vanishes identically (See Corollary $2$ of Lemma $33$ in \cite{HC}). Therefore, we obtain $[E(\phi, \lambda_0)_{Q}]_{\vert e'(Q)} = (0)$.

Now we consider the case that the parabolic $\Q$-subgroup $\sQ$ is associated to the  standard parabolic $\sP_J$ we started with. 
The following assertion concerning the associate class of a given block-$2$ parabolic subgroup $\sP_J$ is a simple observation, based on interpreting an element in $\sL_J$ as a diagonal matrix with block entries of size at most two.
\begin{lemma}\label{block2}
Let $\sP_J$ and $\sP_{J'}$ be two block-$2$ parabolic $\Q$-subgroups of $\sG$ such that $\vert J \vert = \vert J' \vert$, that is, the number of blocks of size two in the corresponding Levi subgroups is the same, then $P_J$ and $P_{J'}$ are in the same associate class.
\end{lemma}
In particular, if there exists $J' \neq J$, with $\vert J \vert = \vert J' \vert$, we see that the associate class $\{\sP_J\}$ contains parabolic $\Q$-subgroups which are not $\sG(\Q)$-conjugate to $\sP_J$. This is already the case if $n = 3$, and $J = \{\alpha_1\}$, $J' = \{\alpha_2\}$.

\begin{theorem}\label{Eisconstant}
Let $[E(\phi, \lambda_0)]$ be a non-zero class in $H^{\mathrm{cd}(\Gamma)}(\Gamma \setminus X, E)$ which is represented by the closed, harmonic differental form $E(\phi, \lambda_0)$ and which originates with a non-zero   cuspidal cohomology class  $[\phi]$ in $H^*_{\mathrm{cusp}}(e'(P_J), E)$
of type $(\pi, w_{P_J})$ where $\pi$ is an irreducible unitary representation of $^0 M_J$ and $w_{P_J}$ denotes the longest element in $W^{P_J}$. Let $\sQ$ be a proper parabolic $\Q$-subgroup of $\sG$. We suppose  that $\sQ$ is associate to $\sP_J$ but $\sQ$ and $\sP_J$ are not $\Gamma$-conjugate.

Then 
\begin{equation}
r^{\textrm{cd}(\Gamma)}_Q( [E(\phi, \lambda_0)]) = 0
\end{equation}
\end{theorem}

\begin{proof}Since the parabolic subgroups $\sP_J$ and $\sQ$ are in the same associate class we have (cf. Section \ref{eisen})
 \begin{equation}
  r^*_Q([E(\phi, \lambda_0)]) = \sum_{s \in W(A_{P_J}, A_Q)} [c(s, \lambda_0)_{s\lambda_0}(\phi_{\lambda_0})]_{\vert e'(Q)}
   \end{equation}
   where $c(s, \lambda_0)_{s\lambda_0}: \Omega^*(\Gamma_{P_J} \setminus X, E) \longrightarrow \Omega^*(\Gamma_Q \setminus X, E)$ is a certain
   ``intertwining" operator defined in \cite[4. 10]{S2}.  It  arises from the corresponding intertwining operator which occurs in the constant Fourier coefficient of the Eisenstein series in question along the parabolic $\sQ$.
   
   First, we consider the case that $\sQ$ is conjugate under $\sG(\Q)$ to $\sP_J$. Thus, there is $g \in \sG(\Q)$ with $\sP_J^g = \sQ$ and $A_J^g = A_Q$ is a split component of $Q$, and $W(A_{P_J}, A_Q)$ is non-trivial.
   Let $s \in W(A_{P_J}, A_Q)$. Then, using Lemma 106 and its Corollary in \cite{HC}, [see also Lemma 4.5.(ii) in \cite{La1}], the intertwining operator is identically zero unless $\sQ$ is conjugate under $\Gamma$ to $\sP_J$. However, this case is excluded by our assumption.
   
   Second, we consider the case that $\sQ \in \{\sP_J\}$ but $\sQ$ is not conjugate under $\sG(\Q)$ to $\sP_J$. Then there is a subset $J' \subset \Delta_{\Q}$, $J' \neq J$ with $\vert J' \vert = \vert J \vert$ such that $\sQ$ is conjugate under $\sG(\Q)$ to $\sP_{J'}$. We may assume  $\sQ = \sP_{J'}$, and, for the sake of simplicity, we write $\sP = \sP_J$.
Recall that both parabolic $\Q$-subgroups $\sQ$ and $\sP$ are block-2 parabolics. 

 Given an element  $s \in W(A_P, A_Q)$, 
the cohomology class $[c(s, \lambda_0)_{s\lambda_0}(\phi_{\lambda_0})]_{\vert e'(Q)}$ in $H^*_{\mathrm{cusp}}(e'(Q), E)$ is a class of weight 
$v_s(\Lambda + \rho) - \rho_{\vert \frak{a}_Q}$ where $v_s \in  W^Q$ is a uniquely determined element with
\begin{equation}
v_s(\Lambda + \rho) - \rho_{\vert \frak{a}_Q} + s\lambda_0 = 0 \quad \text{and}\quad  \chi_{^s\pi} = \chi_{- v_s(\Lambda + \rho)\vert_{\frak{b}_{Q, \C}}}.
\end{equation}

A straightforward computation shows that the only element in $W^Q$ which can satisfy the former condition is $v_s = 1$; see pages $128-129$, \cite{S2}, for a specific case. Note that the corresponding cohomology class $[c(s, \lambda_0)_{s\lambda_0}(\phi_{\lambda_0})]_{\vert e'(Q)}$ is a class of degree $\mathrm{cd}(\Gamma)$, that is, it is an element in 
$H^{\mathrm{cd}(\Gamma)}_{\mathrm{cusp}}(e'(Q), E) = H^{\mathrm{cd}(\Gamma) - \ell(w_Q)}_{\mathrm{cusp}}(\Gamma_M \setminus X_M, F_{\mu_{w_Q}})$.
 However, there is no non-trivial class of the required weight.  Consequently, $r^{\textrm{cd}(\Gamma)}_Q( [E(\phi, \lambda_0)] = 0$.
\end{proof}

Let $J = \{\alpha_{i_1}. \ldots, \alpha_{i_r} \} \subset \Delta_{\Q}$ be a non-empty set of simple roots  
subject to the conditions that if $\alpha_i, \alpha_j \in J$ then $\vert i - j \vert \geq 2.$, and let $\sP_J$ be the corresponding parabolic $\Q$-subgroup of $\sG/\Q = SL_n/\Q$. The parabolic rank of $\sP_J$ is equal $(n-1) - r$. Let $\{\sP_J\}$ be the corresponding associate class of parabolic $\Q$-subgroups of $\sG$. By Lemma \ref{block2}, the class $\{\sP_J\}$ consists of all parabolic $\Q$-subgroups $\sP_{J'}$ of $\sG$ and their $\sG(\Q)$-conjugates where $\vert J' \vert = \vert J \vert$ and $\sP_{J'}$ is a block-2 parabolic. 

Given an arithmetic torsion-free subgroup $\Gamma \subset \sG(\Q)$, the $\Gamma$-conjugacy classes of elements in $\{\sP_J\}$ are in one-to-one correspondence to the faces $e'(Q)$ in $\partial(\Gamma \setminus \overline{X})$ where $\sQ$ ranges over a set of representatives for $\{\sP_J\}/\Gamma$. Given such a representative, say a  parabolic $\Q$-subgroup $\sQ$, using the construction in Theorem \ref{Eisconstruct}, there is a corresponding subspace $H^{\mathrm{cd}(\Gamma)}(\Gamma \setminus X; E)_{e'(Q)}$ of non-zero Eisenstein cohomology classes $[E(\phi, \lambda_0)]$ where $[\phi]$ ranges over all non-zero cuspidal classes in $H^{\mathrm{cd}(\Gamma) - \ell(w_Q)}(\Gamma_M \setminus X_M, F_{\mu_{w_Q}})$. 
 
Let  $\sQ$ and $\sR$ be  representatives of two different $\Gamma$-conjugacy classes in $\{\sP_J\}$. Then Theorem \ref{Eisconstant} implies that the intersection of the corresponding spaces of Eisenstein cohomology classes
\begin{equation}
H^{\mathrm{cd}(\Gamma)}(\Gamma \setminus X; E)_{e'(Q)} \cap H^{\mathrm{cd}(\Gamma)}(\Gamma \setminus X; E)_{e'(R)} = \{0\}
\end{equation}
is trivial. Therefore, the subspace $H^{\mathrm{cd}(\Gamma)}_{\{\sP_J\}}(\Gamma \setminus X; E)$ of Eisenstein cohomology classes originating from  the faces $e'(Q)$ where $\sQ$ ranges of the $\Gamma$-conjugacy classes in $\{\sP_J\}$ is the finite direct sum
\begin{equation}\label{deass}
H^{\mathrm{cd}(\Gamma)}_{\{\sP_J\}}(\Gamma \setminus X; E) = \bigoplus_{\sQ \in \{\sP_J\}/\Gamma}  H^{\mathrm{cd}(\Gamma)}(\Gamma \setminus X; E)_{e'(Q)}.
\end{equation}

Clearly, this space is generated by regular Eisenstein cohomology classes.  If at least one of the origins $H^{\mathrm{cd}(\Gamma)}_{\mathrm{cusp}}(e'(Q), E)$, $\sQ \in \{\sP_J\}$, is non-zero, then also the total subspace $H^{\mathrm{cd}(\Gamma)}_{\{\sP_J\}}(\Gamma \setminus X; E)$ is  non-zero.

The decomposition of the latter space, as exhibited in (\ref{deass}), can be simplified by subdividing the set $ \{\sP_J\}/\Gamma$ which parametrizes the individual summands.
To simplify matters, we now assume that the coefficient system is given by the trivial representation.

If two parabolic $\Q$-subgroups $\sQ$ and $\sR$ of $\sG$ are conjugate under $\sG(\Q)$ then they lie in the same associate class. Evidently, the converse is not correct. Thus, a given associate class $\{\sQ\}$ of parabolic $\Q$-subgroups of $\sG$ falls into a finite number of $\sG(\Q)$-conjugacy classes. Given an arithmetic subgroup $\Gamma \subset \sG$, each of these $\sG(\Q)$-conjugacy classes decomposes into a finite set of $\Gamma$-conjugacy classes. In the case of the previously considered block-2  parabolic $\Q$-subgroups we have the following result regarding the number of $\sG(\Q)$-conjugacy classes within an associate class $\{\sQ\}$. 
\begin{lemma}
Let $\sP$ be a parabolic $\Q$-subgroup of $\sG/\Q = SL_n/\Q$, $n \geq 3$, whose $\sG(\Q)$-conjugacy class is represented by a standard parabolic $\Q$-subgroup $\sP_J$ indexed by the  set $J = \{\alpha_{i_1}. \ldots, \alpha_{i_r} \} \subset \Delta_{\Q}$, $J \neq \emptyset$ subject to the conditions that if $\alpha_i, \alpha_j \in J$ then $\vert i - j \vert \geq 2.$ Let $\mathrm{conj}_{\sG}[\{\sP\}]$ denote  the number of  $\sG(\Q)$-conjugacy classes of parabolic $\Q$-subgroups in the associate class $\{\sP\}$. Then we have
\begin{equation}
\mathrm{conj}_{\sG}[\{\sP\}] = \binom{n - r}{r}.
\end{equation}
\end{lemma}
\begin{proof}Using Lemma \ref{block2} this problem amounts to determine the cardinality of the family $\mathcal{A}$ of subsets $I \subset \Delta_{\Q}$ which are subject to the analogous conditions as $J$ and $\vert J \vert = \vert I \vert$. Consider the set $A = \{1, \ldots, n-1\}$ of indices of the elements in $\Delta_{\Q}$. Given a subset $T \subset A$ we enumerate its elements in increasing order
\begin{equation*}
1 \leq t_1 < t_2 < \ldots < t_r \leq n-1.
\end{equation*}
Since the given conditions are read as $t_i > t_{i-1} + 1$, $i = 1, \ldots, r$, we can transform this sequence into the sequence
\begin{equation*}
1 \leq t_1 < t_2 - 1 < \ldots < t_r  - (r-1) \leq n-1 - (r-1) = n - r.
\end{equation*}
Conversely, given a sequence $1 \leq c_1 < c_2 \ldots < c_r \leq n - r$, we can define an increasing sequence 
\begin{equation*}
1 \leq c_1 < c_2 + 1 < \ldots < c_r  + (r-1) \leq n - r + (r-1) = n-1
\end{equation*}
which is subject to the conditions as given.  Therefore, there is a one-to-one correspondence between the elements in $\mathcal{A}$ and the family of subsets $C$ of $\{1, \ldots, n-r\}$ with $r$ elements. Thus, the cardinality of $\mathcal{A}$ equals $ \binom{n - r}{r}$.
\end{proof}

\begin{proposition}\label{nonzero}
Let $\sP$ be a parabolic $\Q$-subgroup of $\sG/\Q = SL_n/\Q$, $n \geq 3$, whose $\sG(\Q)$-conjugacy class is represented by a standard parabolic $\Q$-subgroup $\sP_J$ determined  by the  non-empty set $J = \{\alpha_{i_1}. \ldots, \alpha_{i_r} \} \subset \Delta_{\Q}$, subject to the conditions that if $\alpha_i, \alpha_j \in J$ then $\vert i - j \vert \geq 2.$ 

Let  $\Gamma = \Gamma(m) \subset \sG(\Q)$ be a principal congruence  subgroup, $m \geq 5$. Then the cuspidal cohomology 
\begin{equation}
H^{\mathrm{cd}(\Gamma)}_{\mathrm{cusp}}(e'(P_{J}), \C) = H^{\mathrm{cd}(\Gamma) - \ell(w_{P_{J}})}_{\mathrm{cusp}}(\Gamma_M \setminus X_M, F_{\mu_{w_{P_{J}}}}) \neq \{0\}.
\end{equation}
does not vanish.
\end{proposition}

\begin{proof} Without loss of generality we may assume that $\sP = \sP_J$. First, consider the case where $J$ consists of one simple root, say $\alpha_j \in \Delta_{\Q}$. Then 
\begin{equation}
^0M = \{\diag(\pm 1,\ldots, A, \ldots, \pm 1)\; \vert \; A \in SL^{\pm}(\R) \}.
\end{equation}
Consequently, since $F_{\mu_{w_{P_{J}}}} \cong \C$ and $ \ell(w_{P_{J}}) = \mathrm{cd}(\Gamma) - 1$,we obtain 
\begin{equation}
H^{\mathrm{cd}(\Gamma)}_{\mathrm{cusp}}(e'(P_{J}), \C) = H^{\mathrm{cd}(\Gamma) - \ell(w_{P_{J}})}_{\mathrm{cusp}}(\Gamma_M \setminus X_M, \C) \cong H^{1}_{\mathrm{cusp}}(\Gamma_M \setminus X_M, \C).
\end{equation}
The latter space is described in terms of relative Lie algebra cohomology groups by
\begin{equation}
H^1_{\mathrm{cusp}}(\Gamma_M \setminus Z_M, \C) = \bigoplus_{(\pi, H_{\pi})} H^{1}(^0\frak{m}, K_M; V_{\pi} \otimes F_{\mu_w})
\end{equation}
where $V_{\pi}$ denotes the isotypic component of the irreducible unitary representation $(\pi, H_{\pi})$ as it occurs in the cuspidal spectrum. The only $(^0\frak{m}, K_M)$-module which can possibly contribute to the right hand side is the irreducible unitary representation which arises from the discrete series representation $D_2$ of the block $SL^{\pm}_2(\R)$ in $^0 M$. The multiplicity with which it occurs in the cuspidal spectrum is, via the Eichler-Shimura isomorphism, equal the dimension $2\cdot \dim_{\C}(S_2(\Delta(m))$ where $S_2(\Delta(m))$ denotes the space of cuspidal automorphic forms of weight two with respect to the congruence group $\Delta(m)$ of $SL_2$. For a given $m > 5$ this space is non-zero.

Second, in the case where $\vert J \vert = r > 1$, $P_{J}$ is a parabolic $\Q$-subgroup whose Levi component $\sL_{J}$ is isomorphic, up to finite index,  to a direct product of $r$ copies of $GL_2/\Q$.  On the real points $GL_2(\R)$ of each of these copies we take the discrete series  representation $D_2$ of $GL_2(\R)$. This gives rise to a representation $\pi$ of $^0M_{J}$, and we have $H^r(^0\frak{m}, K_M, V_{\pi} \otimes \C) \neq \{0\}$. Using the previous step there exist non-zero cuspidal cohomology classes $[\phi] \in H^{\mathrm{cd}(\Gamma)}_{\mathrm{cusp}}(e'(P_{J}), \C) = H^{\mathrm{cd}(\Gamma) - \ell(w_{P_{J}})}_{\mathrm{cusp}}(\Gamma_M \setminus X_M, \C)$. 
\end{proof}

Let $\sP$ be a proper parabolic $\Q$-subgroup of $\sG$.  If $\Gamma \subset SL_n(\Z)$ is a  subgroup of finite index, the   $\sG(\Q)$-conjugacy class of parabolic $\Q$-subgroups of $\sG$ determined by $\sP$ falls into a finite set of $\Gamma$-conjugacy classes. We denote its cardinality by  $\mathrm{conj}_{\Gamma}(\sP)$.

\begin{theorem}\label{main}Let $\sP$ be a parabolic $\Q$-subgroup of $\sG/\Q = SL_n/\Q$, $n \geq 3$, whose $\sG(\Q)$-conjugacy class is represented by a standard parabolic $\Q$-subgroup $\sP_J$ indexed by the  set $J = \{\alpha_{i_1}. \ldots, \alpha_{i_r} \} \subset \Delta_{\Q}$, $J \neq \emptyset$ subject to the conditions that if $\alpha_i, \alpha_j \in J$ then $\vert i - j \vert \geq 2.$ 
Given a torsion-free principal congruence  subgroup $\Gamma = \Gamma(m) \subset \sG(\Q)$, $m > 5$,  the space
\begin{equation}
H^{\mathrm{cd}(\Gamma)}_{\{\sP\}}(\Gamma \setminus X; \C) = \bigoplus_{\sQ \in \{\sP\}/\Gamma}  H^{\mathrm{cd}(\Gamma)}(\Gamma \setminus X; \C)_{e'(Q)},
\end{equation}
generated by  Eisenstein cohomology classes as constructed in Theorem \ref{Eisconstruct} for each of the faces $e'(Q), \sQ \in \{\sP\}$, is non-trivial and its dimension is given by
\begin{equation}
\dim_{\C} H^{\mathrm{cd}(\Gamma)}_{\{\sP\}}(\Gamma \setminus X; \C) = \mathrm{conj}_{\sG}[\{\sP\}] 
 \cdot \mathrm{conj}_{\Gamma}(\sP) \cdot \dim_{\C}H^{\mathrm{cd}(\Gamma)}_{\mathrm{cusp}}(e'(\sP), \C)
\end{equation}
where $\mathrm{conj}_{\sG}[\{\sP\}] = \binom{n - r}{r}$.

\end{theorem}
\begin{proof} Let $P_{J'}$ be a standard parabolic $\Q$-subgroup which belongs to the associate class $\{\sP\}$.  Then $\vert J' \vert = \vert J \vert$, and there is an inner automorphism of $\sG(\Q)$ which defines a $\Q$-rational isomorphism $M_{J} \tilde{\longrightarrow} M_{J'}$. This implies that the corresponding cuspidal cohomology spaces are isomorphic, that is, 
\begin{equation}
H^{\mathrm{cd}(\Gamma)}_{\mathrm{cusp}}(e'(P_{J}), \C) \cong 
H^{\mathrm{cd}(\Gamma)}_{\mathrm{cusp}}(e'(P_{J'}), \C).
\end{equation}
Since these spaces are non-zero by Proposition \ref{nonzero},  it follows that the corresponding Eisenstein cohomology  spaces $H^{\mathrm{cd}(\Gamma)}(\Gamma \setminus X; \C)_{e'(P_{J})}$ and $H^{\mathrm{cd}(\Gamma)}(\Gamma \setminus X; \C)_{e'(P_{J'})}$ are isomorphic and  non-trivial. Therefore the space $H^{\mathrm{cd}(\Gamma)}_{\{\sP\}}(\Gamma \setminus X; \C)$ is non-zero.
The formula for its  dimension follows from the previous argument and arranging the $\Gamma$-conjugacy classes of parabolic $\Q$-subgroups in the associate class $\{\sP\}$ according to the $\sG(\Q)$-cnjugacy class to which they belong.
\end{proof}

If one is interested to determine the size of a specific summand $H^{\mathrm{cd}(\Gamma)}_{\{\sP\}}(\Gamma \setminus X; \C)$ it is necessary to analyze the term $\mathrm{conj}_{\Gamma}(\sP)$ and the cuspidal cohomology $H^{\mathrm{cd}(\Gamma)}_{\mathrm{cusp}}(e'(\sP), \C)$. 

 Let $q = p^{\nu} > 2$ a prime power, and let $\Gamma(q) \subset SL_n(\Z)$ be the principal congruence subgroup of level $q$.
\begin{lemma}\label{Gamclasses}
Let $\sP$ be a parabolic $\Q$-subgroup of $\sG/\Q = SL_n/\Q$, $n \geq 3$, that is $\sG(\Q)$-conjugate to a  standard parabolic $\Q$-subgroup $\sP_J$ indexed by a  set $J = \{\alpha_{i_1}, \ldots, \alpha_{i_r} \} \subset \Delta_{\Q}$, $J \neq \emptyset$, subject to the conditions that if $\alpha_i, \alpha_j \in J$ then $\vert i - j \vert \geq 2.$ Then the number of $\Gamma(q)$-conjugacy classes of $\sP$ is
\begin{equation}
\mathrm{conj}_{\Gamma(q)}(\sP) = \frac{1}{2^{n-r-1}}\cdot \frac{\prod_{i=2}^{n}(q^i - 1)}{(q^2 - 1)^{r}}.
\end{equation}
\end{lemma}

\begin{proof}By definition, the principal congruence subgroup $\Gamma(q)$ is the kernel of the homomorphism $SL_n(\Z) \longrightarrow SL_n(\Z/q\Z)$ given by taking the entries of a matrix $\mod q$. Since this homomorphism is surjective, we have $SL_n(\Z)/\Gamma(q) \cong SL_n(F_q)$ where $F_q = \Z/q\Z$ is the finite field with $q$ elements.

We may assume that $\sP = \sP_J$ where $J = \{\alpha_{i_1}, \ldots, \alpha_{i_r} \} \subset \Delta_{\Q}$, $J \neq \emptyset$ subject to the conditions that if $\alpha_i, \alpha_j \in J$ then $\vert i - j \vert \geq 2.$. Then we have the Levi decomposition $\sP_J = \sL_J \sN_J$, and, accordingly for the group of real points $^0P_J =\; ^0M_J N_J$. Since $\Gamma = \Gamma(q)$ is fixed, we may write
\begin{equation}
\Gamma(q)_{P_J} = \Gamma \cap P_J, \quad \Gamma(q)_{M_J} = \Gamma \cap M_J, \quad \Gamma(q)_{N_J} = \Gamma \cap N_J.
\end{equation}
We note that $\Gamma_{P_J} \subset \;^0P_J =\; ^0M_J N_J$. By the strong approximation property the group $SL_n/\Q$ has the $SL_n(\Z)$-conjugacy class of $\sP_J$ coincides with the $SL_n(\Q)$-conjugacy class of $\sP_J$. Moreover, the normalizer of  $\sP_J$ in $SL_n(\Q)$ is the group $\sP_J$ itself. Therefore, 
\begin{equation}\label{formulaGam}
\mathrm{conj}_{\Gamma(q)}(\sP) = \vert SL_n(\Z)/\Gamma(q) \vert \cdot \vert (SL_n(\Z) \cap {P_J}) / \Gamma(q)_{P_J} \vert^{-1}.
\end{equation}

With regard to the first factor,  the cardinality of the special linear group $SL_n(\Z)/\Gamma(q) \cong SL_n(F_q)$ over the finite field $F_q$ is 
\begin{equation}
\vert SL_n(F_q) \vert = q^{n(n-1)/2} \prod_{i = 2}^{n} (q^i - 1). 
\end{equation}

In order to determine the second factor in formula (\ref{formulaGam}), we may assume, without loss of generality, that $\sP_J$ has the index set $J = \{\alpha_{1}, \alpha_{2 + 1}, \dots,  \alpha_{r +(r-1)} \}$. Then the Levi component of $P_J$  has the form
\begin{equation}
\sL_J = \{\textrm{diag}(A_1, A_3, \dots, A_{2r-1}, a_{2r+1}, \ldots,  a_{n}) \in SL_n(\Q)\; \vert \; A_i \in GL_2(\Q) \; \text{and}\; a_j \in GL_1(\Q)\}. 
\end{equation}
Taking the determinant of each of the components of such a ``diagonal matrix" induces a homomorphism
\begin{equation}
\omega: SL_n(\Z) \cap M_J \longrightarrow (\pm 1)^{n - r - 1},
\end{equation}
formally given by the assignment
\begin{equation*}
 \textrm{diag}(A_1, A_3, \dots, A_{2r-1}, a_{2r+1}, \ldots,  a_{n}) \mapsto (\det A_1, \ldots, \det A_{2r-1}, \det a_{2r+1}, \ldots, \det a_{n-1}).
\end{equation*}
Evidently, the homomorphism $\omega$ is surjective.
Recall  that $SL_n(\Z) \cap P_J \subset \;^0P_J =\; ^0M_J N_J$, and $\Gamma(q) \cap P_J \subset \;^0P_J =\; ^0M_J N_J$. Therefore the  kernel of $\omega$ is
\begin{equation}
\ker \omega = \{\textrm{diag}(A_1, A_3, \dots, A_{2r-1}, a_{2r+1}, \ldots,  a_{n}) \in SL_n(\Z)\; \vert \; A_i \in SL_2(\Z) \; \text{and}\; a_j = 1\}. 
\end{equation}
 Since $q > 2$,  an element $ \textrm{diag}(A_1, A_3, \dots, A_{2r-1}, a_{2r+1}, \ldots,  a_{n}) \in \Gamma(q) \cap P_J$  satisfies $\det A_i = 1$ and $a_j = 1$, thus, it  
 lies in the kernel of $\omega$. It follows that 
\begin{equation}
 \vert (SL_n(\Z) \cap {P_J}) / \Gamma(q)_{P_J} \vert = 2^{n-r-1}\cdot [q (q^2 - 1)]^r \cdot q^{\frac{n(n-1)}{2} - r},
\end{equation}
where $[q (q^2 - 1)] = \vert SL_2(F_q) \vert$, and the last factor is the contribution from the ``unipotent" part $(SL_n(\Z) \cap N_J)/ \Gamma(q)_{N_J}$.
\end{proof}

Ultimately, the non-vanishing Eisenstein cohomology classes in $H^{\mathrm{cd}(\Gamma)}_{\{\sP\}}(\Gamma \setminus X; \C)$ originate from the cuspidal automorphic forms in the space \begin{equation}
H^{\mathrm{cd}(\Gamma)}_{\mathrm{cusp}}(e'(P_{J}), \C) = H^{\mathrm{cd}(\Gamma) - \ell(w_{P_{J}})}_{\mathrm{cusp}}(\Gamma_M \setminus X_M, F_{\mu_{w_{P_{J}}}}) \neq \{0\},
\end{equation}
that is, from elements in the space  $S_2(\Delta(m))$ of cuspidal forms of weight $2$ with respect to  the principal congruence subgroup $\Delta(m)$  of level $m$ in $SL_2$. Since the dimension of this space is known, one can derive, if needed,  quite explicit formulas for the size of $H^{\mathrm{cd}(\Gamma)}_{\{\sP\}}(\Gamma \setminus X; \C)$.

\begin{example}\label{cusptwo} Let $n \geq 3$, and let $\Gamma(q) \subset SL_n(\Q)$, with $q = p^{\nu} > 5$, $p$ a prime.  Take $J = \{\alpha_j\}$ for any simple root in $\Delta_{\Q}$. The Levi subgroup of the corresponding parabolic $\Q$-subgroup $\sP_{J}$ is given as
\begin{equation}
\sL_{J} = \{\textrm{diag}(a_1, \ldots, a_{j-1}, A, a_{j+2}, \ldots, a_n) \in SL_n(\Q) \; \vert \;  A \in GL_2(\Q), a_i \in GL_1(\Q) \}.
\end{equation}
Then we have $H^{\mathrm{cd}(\Gamma(q))}_{\mathrm{cusp}}(e'(P_{J}), \C) \cong H^1_{\mathrm{cusp}}(\Delta(q), \C) \cong S_2(\Delta(q)) \oplus \overline{S_2(\Delta(q))}$. Using Hecke's work \cite{Hecke} the dimension of the latter space is given by
\begin{equation}
\dim_{\C}H^1_{\mathrm{cusp}}(\Delta(q), \C) =  2 g_{\Delta(q)} 
\end{equation}
where $g_{\Delta(q)}$ denotes the  genus of the compact Riemann surface attached  to the group $\Delta(q)$ by adding the cusps to the surface at infinity. The genus is given by the formula
\begin{equation}
g_{\Delta(q)} = 1 + \frac{\mu_q}{12} - (1/2)\cdot [\frac{q^2 - 1}{2}]
\end{equation}
where $\mu_q = (1/2)q(q^2 - 1)$ is the index of $\Delta(q)$ in $SL_2(\Z)/\{\pm 1\}$, and 
the expression in the square brackets is the number of $\Delta(q)$-conjugacy classes of Borel subgroups in $SL_2(\Q)$.
Observe that  $\mathrm{conj}_{\sG}[\{\sP_J\}] = n - 1$. Thus, we are reduced to determine the number  $\mathrm{conj}_{\Gamma(q)}(\sP_J)$ of $\Gamma(p)$-conjugacy classes of parabolic $\Q$-subgroups within the $\sG(\Q)$-conjugacy class of $\sP_J$.
By Lemma \ref{Gamclasses} we have $\mathrm{conj}_{\Gamma(q)}(\sP) = \frac{1}{2^{n-2}}\cdot \prod_{i = 3}^n(q^i - 1)$. Taking the different terms together we obtain
\begin{equation}
\dim_{\C} H^{\mathrm{cd}(\Gamma)}_{\{\sP_J\}}(\Gamma(q) \setminus X; \C) = (n - 1)\cdot \frac{1}{2^{n-2}}\cdot \prod_{i = 3}^n(q^i - 1) \cdot g_{\Delta(q)}.
\end{equation}
where $J = \{\alpha_j\}$ for any simple root in $\Delta_{\Q}$.
\end{example}

\subsection{Mixed cases}\label{Mixed}
We now turn our attention to those cases of parabolic $\Q$-subgroups of $SL_n$ that are in case (II) but not yet covered in the preceding investigation. Recall that the standard parabolic $\Q$-subgroups $\sP_J$ in case (II) are characterized by the subsets $J \subset \Delta_{\Q}$ subject to the condition that if $\alpha_i, \alpha_{i+1} \in J$, with $i \in \{1, \dots, i+2\}$, then $\alpha_{i+2} \notin J$. Previously, in our treatment, we did not discuss how we handle   the occurrence  of blocks of size three, i.e. $GL_3$,  in the Levi component $\L_J$. 
Though one can expect a result analogous to Theorem \ref{main},  exact formulas for the dimension of the cuspidal cohomology $H^{\mathrm{cd}(\Gamma)}_{(\mathrm{cusp})}(e'(P_J), \C)$, for example, in the case of a principal congruence subgroup,  are currently not known. However, the existence of cuspidal cohomology classes for congruence subgroups of sufficiently high level of $SL_3/\Q$ or $GL_3/\Q$ is a consequence of the general results in \cite[Section 3]{LabSchw}, \cite{LabSchwJNT} or \cite[Section 5]{S10}.  

The situation is slightly better in the following case where one can give a lower bound for the dimension in question. Let $\Gamma(p) \subset SL_3(\Q)$ be the principal congruence subgroup of level $p$ where  $p \equiv  3 \mod 8$ and $p \equiv  - 1 \mod 3$. Then, by the main theorem in \cite{LeeS2},
\begin{equation}
\dim_{\C} H^{3}_{\mathrm{cusp}}(\Gamma(p); \C) \geq p (p + 1) \quad \text{where} \quad 3 = \mathrm{cd}(\Gamma(p)).
\end{equation}
It is clear that, following the lines of arguments as given in the case of block-2 parabolic $\Q$-subgroups, a result analogous to Theorem \ref{main} can be obtained. We leave details to the interested reader. 

\begin{example}Let $n \geq 4$, and let $\Gamma \subset SL_n(\Q)$ be a torsion-free congruence subgroup. Take $J = \{\alpha_j, \alpha_{j+1}\}$ for any simple root in $\Delta_{\Q}$. The Levi subgroup of the corresponding parabolic $\Q$-subgroup $\sP_J$ is given as 
\begin{equation}
\sL_{J} = \{\textrm{diag}(a_1, \ldots, a_{j-1}, A, a_{j+3}, \ldots, a_n) \in SL_n(\Q) \; \vert \;  A \in GL_3(\Q), a_i \in GL_1(\Q) \}.
\end{equation}
Then there exists a space 
$H^{\mathrm{cd}(\Gamma)}_{\{\sP_J\}}(\Gamma \setminus X; \C) = \bigoplus_{\sQ \in \{\sP_J\}/\Gamma}  H^{\mathrm{cd}(\Gamma)}(\Gamma \setminus X; \C)_{e'(Q)}$
which is generated by regular Eisenstein cohomology classes  and non-trivial for congruence subgroups of sufficiently high level. Its dimension is given by 
\begin{equation}
\dim_{\C} H^{\mathrm{cd}(\Gamma)}_{\{\sP_J\}}(\Gamma \setminus X; \C) = \mathrm{conj}_{\sG}[\{\sP_J\}]
 \cdot \mathrm{conj}_{\Gamma}(\sP_J) \cdot \dim_{\C}H^{\mathrm{cd}(\Gamma)}_{\mathrm{cusp}}(e'(P_J), \C)
\end{equation}
where $\mathrm{conj}_{\sG}[\{\sP\}] = n-1$.

\end{example}

\section{Construction of Eisenstein cohomology classes -  the  case $(\mathrm{I})$}\label{minpar}

For the sake of completeness we briefly review the results obtained in \cite{S1}, \cite{S1CR} which concern  the subspace $H^{\textrm{cd}(\Gamma)}_{\{\sP_0\}}(\Gamma \setminus X, \C)$ of Eisenstein cohomology classes which originate from  the cohomology of the faces $e'(Q)$, $\sQ \in \{\sP_0\}$, in $\partial(\Gamma \setminus \overline{X})$. Their construction relies on the theory of Eisenstein series in the framework of the adele group attached to the underlying algebraic group $\sG/\Q$; a detailed investigation of certain local and global questions regarding the constant terms of the Eisenstein series is decisive.

\subsection{Cohomology at infinity} We identify the set of non-archimedean places $V_{\Q, f}$ with the set of primes in $\Q$. The group
$\sG = SL_n$ has the strong approximation property (with respect to the set of archimedean places), that is, $\sG(\Q)\sG(\R)$ is dense in $\sG(\A)$, or, equivalently, $\sG(\Q)$ is dense in $\sG(\A_f)$. Let $K = \prod_{v \in V_{\Q, f}} K_v$ be an open compact subgroup of $\sG(\A_f)$. Then we have
\begin{equation}\label{strongapprox}
\sG(\A) = \sG(\Q)\cdot (\sG(\R) \times K).
\end{equation}
Given the open compact subgroup $K \subset \sG(\A_f)$, we define $\Gamma: = \sG(\Q) \cap K$, that is, 
\begin{equation}\label{defGamma}
\Gamma = \{ \gamma \in \sG(\Q)\; \vert \; \gamma \in K_v \; \text{for each} \; v \in V_{\Q, f} \}.
\end{equation}
Then $\Gamma$ is an arithmetic subgroup of $\sG(\Q)$, and we have, using (\ref{strongapprox}) and bringing the maximal compact subgroup $K_{\R}: = SO(n)$ of $\sG(\R)$ into play,
\begin{equation}
\Gamma \setminus \sG(\R) / K_{\R} = \sG(\Q) \setminus \sG(\A) / (K_{\R} \cdot K).
\end{equation}

In particular, the group $L(1): =  \prod_{p \in \Q} SL_n(\Z_p)$ is a maximal open compact subgroup of $\sG(\A_f)$, and, 
given a natural number $m$, the group 
\begin{equation}
L(m): = \{l \in  \prod_{p \in \Q} SL_n(\Z_p)\; \vert \; l \equiv 1 \mod m \}
\end{equation}
is an open compact subgroup of $\sG(\A_f)$. Then the arithmetic group $\Gamma$ defined in (\ref{defGamma}) is the principal congruence subgroup $\Gamma(m)$ of  level $m$ of $SL_n(\Z)$. It is torsion-free for $m \geq 3$. Note that the group $L(m)$ is a normal subgroup of $L(1)$ of finite index.

Let $\sP_0 \subset SL_n/\Q$ be  the  minimal
parabolic $\mathbb{Q}$-subgroup of upper triangular matrices, and $\sP_0 = \sL_0 \sN_0$ its Levi
decomposition where $\sN_0$ is the unipotent radical. The maximal
split torus $\sL_0$ is given as the group $\{\diag(t_1, \ldots, t_n) =:t \in SL_n(\Q)\}$ of diagonal matrices with determinant one. To simplify notation we write $\sT: = \sL_0$ for the maximal $\Q$-split torus $\sL_0$. Note that the associate class $\{\sP_0\}$ consists of all minimal parabolic $\Q$-subgroups of $\sG$. 

It is necessary to deal with all the finitely $\Gamma(m)$-conjugacy classes of minimal parabolic $\Q$-subgroups of $\sG$ at the same time. The face $e'(P_0)$ will be our point  of reference in our adelic description of the cohomology at infinity that depends on the associate class $\{\sP_0\}$.
With regard to the $\sT(\R)$-module structure, the cohomology of the face $e'(P_0)$ decomposes as 
\begin{equation}\label{Nzero}
H^\ast(\mathfrak{n}_0,\C)=\oplus_{w\in W^{P_0}} F_{\mu_w}
\end{equation}
where the sum ranges over the elements $w$ in the Weyl group  $W^{P_0} = W$, and $F_{\mu_w}$
denotes the irreducible finite-dimensional
$\sT(\C)$-module of highest weight $\mu_w = w(\rho_{P_0}) - \rho_{P_0}.$ We choose a harmonic differential form $\phi_w \in \Omega^{\ell(w)}(e'(P_0), \C)$ such that 
$F_{\mu_w} = \C[\phi_w]$, that is, the form $\phi_w$ represents a class of weight $\mu_w$. We denote by $\eta_w:\;  ^0\sT(\R) \longrightarrow \C^{\times}$ the character obtained by restricting $\mu_w$ to $^0\sT(\R)$.  The character $\eta_w, w \in W$, is a sign character. Given the longest element $w_0 \in W$, $\eta_{w_0}$ is the trivial character.

 Given a fixed natural number $m$, we consider the space of double cosets
 \begin{equation}
 \sT(\Q) \setminus \sT(\A)/ \sT(\R)^0 \cdot (\sT(\A_f) \cap L(m)) =\; ^0\sT(\R) \times \sT(\Q) \setminus \sT(\A_f)/ (\sT(\A_f) \cap L(m)).
 \end{equation}
The second factor on the right hand side can be rewritten as
\begin{equation}
\sT(\Q) \setminus \sT(\A_f)/ (\sT(\A_f) \cap L(m)) = \sT(\Z) \setminus (\sT(\A_f) \cap L(1))/ (\sT(\A_f) \cap L(m)) \cong [(\pm 1)\setminus (\Z/m\Z)^{\times}]^{n-1}.
\end{equation}
Consequently, a character on $\sT(\Q) \setminus \sT(\A_f)/ (\sT(\A_f) \cap L(m))$ is uniquely determined by the values it takes on the group $\mathcal{T}_m: = \sT(\A_f) \cap L(1)/\sT(\A_f) \cap L(m)$. 

Let $\mathcal{L}_m$ denote the finite  group $L(1)/L(m)$. The space $\C[\mathcal{L}_m]$ of $\C$-valued functions on $\mathcal{L}_m$ is a $\mathcal{L}_m$-bimodule in a natural way. The group $\mathcal{U}_m: = L(m)\setminus L(m)
(^0\sT(\Q) \cap L(1))\cdot (\sN_0(\A_f) \cap L(1))$ acts via right translation on $\C[\mathcal{L}_m]$. The subspace of elements in $\C[\mathcal{L}_m]$ which are invariant under this action can be identified with $\C[\mathcal{L}_m/\mathcal{U}_m]$. Since $\sT$ normalizes the unipotent radical $\sN_0$, there is a natural action of $\mathcal{T}_m$ on $\C[\mathcal{L}_m/\mathcal{U}_m]$. Note that the cosets which are represented by an element in $^0\sT(\Q) \cap L(1)$ act trivially. With regard to this action of $\mathcal{T}_m$ there is a decomposition $\C[\mathcal{L}_m/\mathcal{U}_m] = \oplus V_{\chi}$ into eigen spaces $V_{\chi}$, where $\chi: \mathcal{T}_m \longrightarrow \C^*$ ranges over the characters trivial on $^0\sT(\Q) \cap L(1)$.
As discussed in the previous paragraph we can interpret the character $\chi$ as a character on  $\sT(\Q) \setminus \sT(\A_f)/ (\sT(\A_f) \cap L(m))$, to be denoted by the same letter.
Finally, by trivial extension, we obtain a quasi character \begin{equation}
\chi: \sT(\Q) \setminus \sT(\A)/ (\sT(\A_f) \cap L(m)) \longrightarrow \C^*.
\end{equation}
 In particular, given $t_{\infty} \in \sT(\R)$, $\chi(t_{\infty}) = 1$.

 Next we have to merge  this approach with  the description of $H^*(e'(P_0), \C)$ as given in the decomposition \ref{Nzero}. The starting point is the following result [see Satz 5.7 in \cite{S1}]
 
  \begin{proposition} Let $D$ be a set of representatives  for the finitely many $\Gamma(m)$-conjugacy classes of minimal parabolic $\Q$-subgroups of $\sG$. Then there is an isomorphism  
 \begin{equation}
 \oplus_{\sQ \in D} H^*(e'(Q), \C) \tilde{\longrightarrow} H^*(e'(P_0), \C) \otimes \C[\mathcal{L}_m/\mathcal{U}_m].
  \end{equation}
  \end{proposition}

Therefore, in view of the decomposition $H^*(e'(P_0), \C) \cong \oplus_{w \in W} F_{\mu_w}$, a cohomology class in $ \oplus_{\sQ \in D} H^*(e'(Q), \C)$ may be characterized by an element $w \in W$ and a  character on $\mathcal{T}_m$, trivially extended to a quasi character $\chi$ on $ \sT(\Q) \setminus \sT(\A)/ (\sT(\A_f) \cap L(m))$. The latter character has to be compatible with a given $w \in W$ in the following sense:
Let $\hat{\sT}_{\A}$ be the set of characters $\chi:  \sT(\Q) \setminus \sT(\A)/ \sT(\R)^0 \cdot (\sT(\A_f) \cap L(m)) \longrightarrow \C^{\times}$.
Given the  sign character $\eta_w:\;  ^0\sT(\R) \longrightarrow \C^{\times}$, indexed by $w \in W$, the character $\chi$ has to be an element of the set $\hat{\sT}_{\A}(\eta_w)$ of characters  in $\hat{\sT}_{\A}$ whose restriction to $^0\sT(\R)$ coincides with $\eta_w$.

\subsection{Eisenstein cohomology classes}
The torus $\sT \subset \sG$ can be written as a direct product of tori $\sT_i$, $i = 1, \ldots, n-1$, and accordingly we write  $\chi = (\chi_1, \ldots, \chi_{n-1})$ for the quasi character $\chi \in \hat{\sT}_{\A}$ If $\chi_i \neq 1$ for all $i = 1, \ldots, n-1$, then, due to the specific choice of $\chi$ in $\hat{\sT}_{\A}$, the quasi character $\chi_i$ on $\A^{\times}/\Q^{\times}$ is not principal, that is, $\chi_i$ is non-trivial on the maximal compact subgroup of the ideles of norm one modulo $\Q^{\times}$. The following result is Theorem 9.1. in \cite{S1}.

\begin{theorem}\label{Eisminimal} Let  $\chi = (\chi_1, \ldots, \chi_{n-1}) \in \hat{\sT}_{\A}$ be a quasi character such that $\chi$ is compatible with the trivial sign character $\eta_{w_0}$ which corresponds to the longest element $w_0 \in W$. Suppose that each component $\chi_i$, $i = 1, \ldots, n-1$, in non-trivial. Then there is a corresponding Eisenstein series $E(\phi_{w_0}, \chi, \lambda)$ attached to the pair $(\phi_{w_0}, \chi)$ which is holomorphic at the point $\lambda_0 = \rho_{P_0}$ and  gives rise to a non-trivial cohomology class in 
$H^{\mathrm{cd}(\Gamma)}(\Gamma \setminus X, \C)$. This class is represented by the closed harmonic differential form given by the value the Eisenstein series 
$E(\phi_{w_0}, \chi, \lambda)$  takes at the point 
$\lambda_0$. The restriction of the class $[E(\phi_{w_0}, \chi, \lambda_0)]$ under the natural restriction map 
\begin{equation}
r^{\mathrm{cd}(\Gamma}_{\{\sP_0\}}: H^{\mathrm{cd}(\Gamma)}(\Gamma \setminus X, \C) \longrightarrow \oplus_{\sQ \in D} H^*(e'(Q), \C) \tilde{\longrightarrow} H^*(e'(P_0), \C) \otimes \C[\mathcal{L}_m/\mathcal{U}_m]
  \end{equation}
  is the class started with.
\end{theorem}

\begin{corollary}\label{Cormin} Let $\Gamma(q) \subset SL_n(\Z)$ be the principal congruence subgroup of level $q = p^{\nu}$, $p \neq 2$ a prime, and $\nu \geq 1$, 
Let $H^{\textrm{cd}(\Gamma)}_{\{\sP_0\}}(\Gamma \setminus X, \C)$ be the space of cohomology classes in $H^{\textrm{cd}(\Gamma)}(\Gamma \setminus X, \C)$ which restrict non-trivially to the cohomology of a face $e'(Q)$, $Q \in \{\sP_0\}$. Then 
\begin{equation}
\dim_{\C}H^{\textrm{cd}(\Gamma)}_{\{\sP_0\}}(\Gamma(q) \setminus X, \C) \geq (\frac{1}{2} (p^{\nu} - 1) - 1)^{n-1}. 
\end{equation}
In particular, if $\nu$ tends to infinity the dimension is unbounded.
\end{corollary}
\begin{proof}
 The Eisenstein cohomology classes $[E(\phi_{w_0}, \chi, \lambda_0)]$ where $\chi = (\chi_1, \ldots, \chi_{n-1})$ ranges over all characters occurring  in the decomposition  $\C[\mathcal{L}_q/\mathcal{U}_q] = \oplus V_{\chi}$ with $\chi_i \neq 1$ for all $i = 1, \ldots, n-1$, generate a subspace in $H^{\textrm{cd}(\Gamma)}_{\{\sP_0\}}(\Gamma \setminus X, \C)$. Since the characters in question are uniquely determined by their restriction $\chi: \mathcal{T}_{p^{\nu}} \longrightarrow \C^{\times}$, a counting of the characters which satisfy the given condition leads to the lower bound as claimed.
\end{proof}

\section{Comparison with prior work and extensions}\label{Conclude} 

\subsection{The adjunction of corners and the Steinberg module} The construction by adjoining corners to  a given arithmetic quotient $\Gamma \setminus X$ as in Section \ref{arithquo} has various  consequences for the cohomology of $\Gamma$, among them,  first, $H^i(\Gamma; \Z[\Gamma]) = 0$ except  in degree $i = \mathrm{d}(G) - \textrm{rk}_{\Q} \sG$, where it is a free module $I$ of infinite rank, and, second, there is an isomorphism
\begin{equation}
H^i(\Gamma; A) \cong H_{(\mathrm{d}(G) - \textrm{rk}_{\Q} \sG) - i}(\Gamma; I \otimes A), \quad i \in \Z,
\end{equation}
for any $\Gamma$-module $A$. In particular, the cohomological dimension of $\Gamma$, denoted $\textrm{cd}(\Gamma)$,  is $\mathrm{d}(G) - \textrm{rk}_{\Q} \sG$. 
The dualizing module $I$ can be realized on the reduced  homology group $\tilde{H}_{\textrm{rk}_{\Q} \sG - 1}(\mathcal{T}_{\sG}, \Z)$ of the Tits building $\mathcal{T}_{\sG}$. The natural action of $\sG(\Q)$ on   
$\mathcal{T}_{\sG}$ induces an action on this homology group such that $st_{\sG}: = \tilde{H}_{\textrm{rk}_{\Q} \sG - 1}(\mathcal{T}_{\sG}, \Z)$ is a module for $\sG(\Q)$, called the Steinberg module.
Consequently, in  the top degree $\textrm{cd}(\Gamma)$, the duality relation (with $A = \Q$) reads 
\begin{equation}
H^{\textrm{cd}(\Gamma)}(\Gamma, \Q) \cong H_0(\Gamma; I \otimes \Q) \cong (st_{\sG})_{\Gamma}
\end{equation}
where the subscript $( - )_{\Gamma}$ denotes the space of coinvariants under $\Gamma$.

 This close relation between the cohomology of an arithmetic subgroup $\Gamma \subset \sG(\Q)$ in the degree of the virtual cohomological dimension and the Steinberg module for $\sG(\Q)$ was used by several authors, among them Lee-Szczarba \cite{LeeSz}, Ash \cite{AAunstable} in the case $SL_n$,  and Reeder \cite{Reeder} for general $\sG$,  to detect non-vanishing cohomology classes in  $H^{\mathrm{cd}(\Gamma)}(\Gamma \setminus X; \C)$.
 More recently, by a close analysis of the Steinberg module, this approach was taken up by Church et al. \cite{CFP} to prove vanishing and non-vanishing results for 
 $H^{\textrm{cd}(\Gamma)}(\Gamma, \Q)$ where $\Gamma = SL_n(\mathcal{O}_k)$,  $\mathcal{O}_k$ the ring of integers in an algebraic number field $k$.
 
 In a  different direction,  in the case of a congruence subgroup $\Gamma(p) \subset SL_n/\Q$ of prime level $p$,  Miller et al. \cite{MPP}  prove that the $\Gamma(p)$-invariant map
$\tilde{H}_{n - 2}(\mathcal{T}_{SL_n}, \Z) \longrightarrow \tilde{H}_{n - 2}(\mathcal{T}_{SL_n}/\Gamma(p), \Z)$, induced by the quotient map $\mathcal{T}_{SL_n} \longrightarrow \mathcal{T}_{SL_n}/\Gamma(p)$, is surjective. Therefore, the size of the right hand side provides a lower bound for the space $(st_{SL_n})_{\Gamma(p)}$, thus, for  the dimension of $H^{\textrm{cd}(\Gamma(p))}(\Gamma(p), \Q)$. 

In the case $n=3$, their result is weaker than the one obtained in \cite{LeeS}, \cite{LeeS2} by using the theory of automorphic forms to construct non-trivial cohomology classes in $H^*(\Gamma, \C)$, where $\Gamma$ is a torsion-free congruence subgroup of $SL_3/\Q$. 
This also  the case for arbitrary $n$ as we show now. 

\subsection{Lower bounds for the dimension of $H^{\mathrm{cd}(\Gamma(q))}(\Gamma(q) \setminus X; \C)$} 
If we put together the various results in Sections \ref{paramatter},  \ref{ParaII}, and \ref{minpar} we can derive lower bounds for the dimension of the top degree cohomology group $H^{\mathrm{cd}(\Gamma(q))}(\Gamma(q) \setminus X; \C)$ of a given
principal congruence subgroup $\Gamma(q) \subset SL_n(\Z)$ of level $q = p^{\nu}$, $p \neq 2$ a prime, and $\nu \geq 1$. Our starting point is the decomposition
 \begin{equation}
H^*(\Gamma(q) \setminus X, \C) = H^*_{\textrm{cusp}}(\Gamma(q) \setminus X, \C) \oplus \bigoplus_{\{\sP\} \neq \{\sG\}}H^*_{\{\sP\}}(\Gamma(q) \setminus X, \C).
\end{equation}
of $H^*(\Gamma(q) \setminus X, \C)$ into the cuspidal cohomology and the Eisenstein cohomology, defined to be
\begin{equation} 
H^*_{\textrm{Eis}}(\Gamma(q) \setminus X, \C): =  \bigoplus_{\{\sP\} \neq \{\sG\}}H^*_{\{\sP\}}(\Gamma(q) \setminus X, \C).
\end{equation}

For $n > 3$,  the cuspidal cohomology $H^*_{\textrm{cusp}}(\Gamma(q) \setminus X, \C)$ vanishes in degree $\mathrm{cd}(\Gamma(q)) = \frac{n(n-1)}{2}$, as pointed out in Section \ref{arithquo}. In the case $n = 2$, one has precise knowledge about $H^1_{\mathrm{cusp}}(\Gamma(q) \setminus, \C)$, $\Gamma(q) \subset SL_2(\Z)$, see Example \ref{cusptwo}. For $n = 3$, there are only some non-vanishing results but no exact formulas for the dimension in degree $\mathrm{cd}(\Gamma(q)) = 3$, $\Gamma(q) \subset SL_3(\Z)$, see \cite[Section 3]{LabSchw}, \cite{LabSchwJNT} or \cite{LeeS2}.

With regard to the Eisenstein cohomology, using Proposition \ref{propparamatter}, we have 
\begin{equation}
H^{\mathrm{cd}(\Gamma(q))}_{\mathrm{Eis}}(\Gamma(q) \setminus X; \C) =   H^{\textrm{cd}(\Gamma)}_{\{\sP_0\}}(\Gamma(q) \setminus X, \C) \oplus \bigoplus_{\{\sP\} \in \mathrm{Ass}_{\sG}^{\mathrm{cd}}}
H^{\textrm{cd}(\Gamma)}_{\{\sP\}}(\Gamma(q) \setminus X, \C) 
\end{equation}
where $\mathrm{Ass}_{\sG}^{\mathrm{cd}}$ denotes the finite set of classes of associate proper standard parabolic $\Q$-subgroups $\sP_J$ whose defining set $J \subset \Delta_{\Q}$ belongs to the family $\mathcal{J}^{\mathrm{cd}}$. Note that an associate class $\{\sP\}$ falls into $\sG(\Q)$-conjugacy classes of parabolic $\Q$-subgroups, thus, the defining set $J$ is not uniquely determined by $\{\sP\}$.

For each of the summands  $H^{\textrm{cd}(\Gamma)}_{\{\sP\}}(\Gamma(q) \setminus X, \C)$ with $\{\sP\} \in \mathrm{Ass}_{\sG}^{\mathrm{cd}}$
there is  the formula
\begin{equation}
\dim_{\C} H^{\mathrm{cd}(\Gamma)}_{\{\sP_J\}}(\Gamma \setminus X; \C) = \mathrm{conj}_{\sG}[\{\sP_J\}]
 \cdot \mathrm{conj}_{\Gamma}(\sP_J) \cdot \dim_{\C}H^{\mathrm{cd}(\Gamma)}_{\mathrm{cusp}}(e'(P_J), \C)
\end{equation}
where $\sP_J \in \{\sP\}$. Thus, by taking together the lower bound for $\dim_{\C}H^{\textrm{cd}(\Gamma)}_{\{\sP_0\}}(\Gamma(q) \setminus X, \C)$, given in Corollary \ref{Cormin},  and the results in Sections  \ref{paramatter},  \ref{ParaII} we obtain a lower bound for the cohomology group in the top degree.

We exemplify the procedure in the following example:
The set  $\mathrm{Ass}_{\sG}^{\mathrm{cd}}$ attached to the group $\sG = SL_5/\Q$ consists of four classes of associate proper parabolic $\Q$-subgroups. First, one consists of the four $\sG(\Q)$-conjugacy classes represented by the standard parabolic subgroups $\sP_J$ of parabolic rank three,  that is, where $J = \{\alpha_1\}, \{\alpha_2\}, \{\alpha_3\},$ and $\{\alpha_4\}$. Second, the standard parabolic $\Q$-subgroups of parabolic rank two fall into two associate classes, one, say $\{\sQ_2\}$, consists of the block-2 parabolic subgroups whereas the other one, say $\{\sQ_3\}$, consists of the block-3 parabolic subgroups. Finally, there is the mixed case, denoted $\{\sR\}$; it falls into  two different $\sG(\Q)$-conjugacy classes, namely, the one of $\sP_J$ with $J = \{\alpha_1, \alpha_2, \alpha_4\}$, and the one of $\sP_{J'}$ with $J' = \{\alpha_1, \alpha_3, \alpha_4\}$. The two other maximal standard parabolic $\Q$-subgroups do not account for  a class in $\mathrm{Ass}_{\sG}^{\mathrm{cd}}$. Then, given $\Gamma(q) \subset SL_5(\Q)$, we obtain the estimate

\begin{equation}
\dim_{\C}H^{\mathrm{cd}(\Gamma(q))}(\Gamma(q) \setminus X; \C) \geq  (\frac{1}{2} (q - 2))^4 + 4\cdot \frac{1}{2^3} \prod_{i=3}^5(q^i - 1) g_{\Delta(q)} + RS_5(q)
\end{equation}
where $RS_5(q)$ denotes the sum of the remaining summands
\begin{equation}
\dim_{\C}H^{\textrm{cd}(\Gamma)}_{\{\sQ_2\}}(\Gamma(q) \setminus X, \C), \quad \dim_{\C}H^{\textrm{cd}(\Gamma)}_{\{\sQ_3\}}(\Gamma(q) \setminus X, \C),\quad  \text{and} \dim_{\C}H^{\textrm{cd}(\Gamma)}_{\{\sR\}}(\Gamma(q) \setminus X, \C). 
\end{equation}
A formula, depending only on $q$, for each of these terms requires explicit knowledge of the cuspidal cohomology groups $H^{\mathrm{cd}(\Gamma(q))}_{\mathrm{cusp}}(e'(P), \C)$, thus,
at least for the latter two cases,  a dimension formula for spaces of cohomological cuspidal automorphic forms for congruence subgroups of $SL_3(\Q)$. This is currently out of reach. However, as mentioned before, one has non-vanishing results for congruence groups of sufficiently high level, see \cite{LabSchw}, \cite{LabSchwJNT}, \cite{LeeS2}.
\begin{remark}For arbitrary $n \geq 3$ and $\Gamma(q) \subset SL_n(\Q)$, $q = p^{\nu} > 2$, $p$ a prime, the lower bound for the dimension of the cohomology of $\Gamma(q)$ in the top degree would read as
\begin{equation}\label{dim}
 \dim_{\C}H^{\mathrm{cd}(\Gamma(q))}(\Gamma(q) \setminus X; \C) \geq   (\frac{1}{2} (q - 1) - 1)^{n-1} + (n - 1)\cdot \frac{1}{2^{n-2}}\cdot \prod_{i = 3}^n(q^i - 1) \cdot g_{\Delta(q)} + RS_n(q).
\end{equation}
where $RS_n(q)$ has the analogous meaning.
Note that the lower bound given by formula (\ref{dim}), taken without the term $RS_n(q)$, exceeds the one given  in \cite[Section 4]{MPP}.
\end{remark}

\subsection{Other groups and fields} Let $\Gamma$ be a torsion-free  arithmetic subgroup of a connected semi-simple algebraic $\Q$-group $\sG$ of positive $\Q$-rank. For the sake simplicity we assume that $\sG$ is $\Q$-split and belongs to the family of classical groups, meaning, beside general linear groups,  symplectic groups, orthogonal groups or unitary groups.
Given a   finite-dimensional irreducible representation $(\nu, E)$ of the real Lie group $G = \sG(\R)$ on a complex vector space, the  cohomology $H^*(\Gamma \setminus X, E)$ of the locally symmetric space $\Gamma \setminus X$ has  a direct sum decomposition,
\begin{equation}
H^*(\Gamma \setminus X, E) = H^*_{\textrm{cusp}}(\Gamma \setminus X, E) \oplus \bigoplus_{\{\sP\}}H^*_{\{\sP\}}(\Gamma \setminus X, E)
\end{equation}
into the subspace of classes represented by cuspidal automorphic forms for $\sG$ with respect to $\Gamma$ and the  Eisenstein cohomology. The latter space is  decomposed according to the 
classes $\{\sP\}$ of associate proper parabolic $\Q$-subgroups of $\sG$. Each summand is built up by Eisenstein series (derivatives or residues of such) attached to cuspidal automorphic representations  $\pi$ on the Levi components of elements in $\{\sP\}$. 

Being interested in the cohomology group $H^{\mathrm{cd}(\Gamma)}(\Gamma \setminus X, E)$ in the degree of the cohomological dimension of $\Gamma$, it should be possible to determine which associate classes $\{\sP\}$ of parabolic $\Q$-subgroups of $\sG$ finally contribute to this decomposition. Recall that a parabolic $\Q$-subgroup is the semi-direct product of its unipotent radical and any Levi subgroup. These Levi subgroups are products of groups of $GL$-type and of isometry groups.

For example, in the case of the symplectic group $Sp_n/\Q$ of rank $n$, up to $\sG(\Q)$-conjugay,  a maximal parabolic $\Q$-subgroup has the form $\sP_r \cong \sL_r \sN_r$, $r = 1, \ldots, n$, with Levi subgroup $\sL_r \cong  GL_r \times Sp_{n-r}$ if $r < n$, and $\sL_n \cong GL_n$ if $r = n$, and $\sN_r$ is the unipotent radical. Note that such a parabolic $\Q$-subgroup is conjugate to its opposite parabolic subgroup, thus, the associate class $\{\sP_r\}$ coincides with the $\sG(\Q)$-conjugacy class $\sP_r$. All other standard  parabolic $\Q$-subgroups are expressible as intersections of the proper standard maximal parabolic $\Q$-subgroups.

A result, analogous to Proposition \ref{propparamatter}, hinges on an explicit knowledge of the range in which the Levi subgroup of a given parabolic $\Q$-subgroup $\sP$ of $\sG$ has possibly  non-vanishing cuspidal cohomology, that is, $H^*_{\mathrm{cusp}}(e'(P), E)$ is non-trivial. In view of the final construction of Eisenstein cohomology classes, this requires a thorough understanding of the cuspidal cohomology of arithmetic groups in classical groups. We refer to \cite{LS4} and \cite{S7}.

With regard to the existence of Eisenstein cohomology classes in $H^*_{\{\sP_0\}}(\Gamma \setminus X, E)$ where $\{\sP_0\}$ denotes the associate class minimal parabolic $\Q$-subgroups of $\sG$ one has to expect a result similar to Theorem \ref{Eisminimal} and its Corollary. In the case $\sG = Sp_n$, see \cite{Lai}, and, in the case of a generic coefficient system, see \cite{S4}.

 If $\Gamma$ is an  arithmetic subgroup of a connected semi-simple algebraic group $\sG$ defined over some algebraic number field $k$ and  of positive $k$-rank one can pose analogous questions. Work of Harder in the case $GL_2/k$ \cite{Ha3} stands as initial guideline, see also \cite{Ha4} or \cite{S10}.

\bibliographystyle{amsalpha}

\end{document}